\documentclass[a4paper,11pt]{article}
\usepackage{xargs}
\usepackage[pdftex,dvipsnames]{xcolor}
\usepackage{comment,amsmath,amssymb,amsthm,pstricks,changepage,pifont,graphicx,titlesec,wrapfig,faktor}
\usepackage[english]{babel}
\usepackage[nomessages]{fp}
\usepackage{algorithm}
\usepackage[noend]{algpseudocode}
\usepackage{hyperref}

\newtheorem{theorem}{Theorem}

\newtheorem{lemma}[theorem]{Lemma}

\newtheorem{problem}[theorem]{Problem}

\theoremstyle{definition}
\newtheorem{algo}[theorem]{Algorithm}
\theoremstyle{remark}
\newtheorem{heuristic}[theorem]{Heuristic}

\usepackage[colorinlistoftodos,prependcaption,textsize=tiny]{todonotes}
\newcommandx{\unsure}[2][1=]{\todo[linecolor=red,backgroundcolor=red!25,bordercolor=red,#1]{#2}}
\newcommandx{\info}[2][1=]{\todo[linecolor=OliveGreen,backgroundcolor=OliveGreen!25,bordercolor=OliveGreen,#1]{#2}}

\newcommand\blfootnote[1]{%
  \begingroup
  \renewcommand\thefootnote{}\footnote{#1}%
  \addtocounter{footnote}{-1}%
  \endgroup
}

\hypersetup{
    colorlinks,
    linkcolor={red!50!black},
    citecolor={blue!50!black},
    urlcolor={blue!80!black}
}

\allowdisplaybreaks

\theoremstyle{remark}
\newtheorem{remark}[theorem]{Remark}
\theoremstyle{definition}

\algnewcommand\Input{\item[\textbf{Input}]}
\algnewcommand\Output{\item[\textbf{Output}]}
\algnewcommand\Outputwhite{\item[\textbf{\textcolor{white}{Output}}]}
\usepackage{float}

\newcommand*\varhrulefill[1][0.4pt]{\leavevmode\leaders\hrule height#1\hfill\kern0pt}

\DeclareMathOperator{\GL}{GL}
\DeclareMathOperator{\lwidth}{lw}
\DeclareMathOperator{\conv}{conv}

\DeclareMathOperator{\pdet}{pdet}
\DeclareMathOperator{\chr}{char}
\DeclareMathOperator{\spec}{Spec}
\DeclareMathOperator{\proj}{Proj}

\DeclareMathOperator{\QQ}{\mathbf{Q}}
\DeclareMathOperator{\RR}{\mathbf{R}}
\DeclareMathOperator{\CC}{\mathbf{C}}
\DeclareMathOperator{\FF}{\mathbb{F}}
\DeclareMathOperator{\PPq}{\mathbb{P}}
\DeclareMathOperator{\PPK}{\mathbf{P}}
\DeclareMathOperator{\AAq}{\mathbb{A}}
\DeclareMathOperator{\AAK}{\mathbf{A}}
\DeclareMathOperator{\TTq}{\mathbb{T}}
\DeclareMathOperator{\TTK}{\mathbf{T}}
\DeclareMathOperator{\ZZ}{\mathbf{Z}}

\newcommand{\polfig}[4]{
\qquad 
\begin{minipage}[b]{#2cm}
\begin{center}
  \includegraphics[height=#3cm]{#1}
\end{center}
\end{minipage}
\FPeval{\result}{{#3}/2 - 0.2}
\begin{minipage}[b]{2cm}
(#4)
\vspace{\result cm}
\end{minipage}
} 

\begin{document}

\nocite{*}

\title{Point counting on curves using a gonality preserving lift}
\author{Wouter Castryck and Jan Tuitman}
\date{}
\maketitle

\begin{abstract} 
We study the problem of lifting curves from finite fields to number fields in 
a genus and gonality preserving way. More precisely, we sketch how this can be done efficiently for
curves of gonality at most four, with an in-depth treatment 
of curves of genus at most five over finite fields of odd characteristic, including an implementation in Magma. We then use such a lift as input
to an algorithm due to the second author for computing zeta functions of curves over finite fields using $p$-adic cohomology. 
\end{abstract}


\section{Introduction}

This article is about efficiently lifting algebraic curves over finite fields to characteristic zero, in a genus and gonality preserving way,
with an application to $p$-adic point counting.
Throughout, our curves are always understood to be geometrically irreducible, but not necessarily non-singular and/or complete. 
By the genus of a
curve we mean its geometric genus, unless otherwise stated. As for the gonality of a curve over a field $k$, we make a distinction between two notions: 
by its \emph{$k$-gonality} we mean the minimal degree of a non-constant $k$-rational map to the projective line, while by its \emph{geometric gonality}
we mean the $\bar{k}$-gonality, where $\bar{k}$ denotes an algebraic closure of $k$. We also make a notational distinction between projective, affine or toric (= affine minus coordinate hyperplanes) $n$-space in characteristic zero, in which case we write $\PPK^n, \AAK^n, \TTK^n$, 
and their finite characteristic counterparts, where we opt for $\PPq^n, \AAq^n, \TTq^n$. Apart from that we avoid reference to the base field, which should always be clear from the context.
Similarly we write $\QQ$ for the field of rational numbers and $\FF_q$ for the finite field with $q$ elements, where $q$ is a power of a prime number $p$.
For each such $q$ we fix a degree $\log_pq$ extension $K \supset \QQ$ in which $p$ is inert, and let $\mathcal{O}_K$ denote its ring of integers. We then identify $\FF_q$ with the residue field $\mathcal{O}_K / (p)$.
Our lifting problem is as follows:

\begin{problem} \label{liftingproblem}
Given a curve $\overline{C}$ over $\FF_q$, find an efficient algorithmic way of producing a polynomial $f \in \mathcal{O}_K[x,y]$ such that 
\begin{enumerate}
  \item[(i)] its reduction mod $p$ defines a curve that is birationally equivalent to $\overline{C}$, 
  \item[(ii)] the curve $C \subset \AAK^2$ it defines has the same genus as $\overline{C}$, 
  \item[(iii)] its degree in $y$ equals the $\FF_q$-gonality of $\overline{C}$.
\end{enumerate}
\end{problem}
Note that these conditions imply that the $K$-gonality of $C$ equals the $\FF_q$-gonality of $\overline{C}$, because
the gonality cannot increase under reduction mod $p$; see e.g.\ \cite[Thm.\,2.5]{derickx}.
We are unaware of whether an $f$ satisfying (i-iii) exists in general. Grothendieck's existence 
theorem~\cite{illusie} implies that in theory one can achieve (i) and (ii) over the ring of integers $\ZZ_q$ of the 
$p$-adic completion $\QQ_q$ of $K$, but, firstly, it is not clear that we can always take $f$ to be defined over $\mathcal{O}_K$ 
and, secondly, we do not know whether it is always possible to incorporate (iii), let alone in an effective way. To give a concrete 
open case, we did not succeed in dealing with Problem~\ref{liftingproblem} for curves of genus four having $\FF_q$-gonality five, which can only exist if $q \leq 7$. (However, as we will see, among all curves of genus at most five, the only cases that we cannot handle are 
pathological examples of the foregoing kind.)

We are intentionally vague about what it means to be \emph{given} a curve $\overline{C}$ over $\FF_q$. 
It could mean that we are considering the affine plane curve defined by a given 
absolutely irreducible polynomial $\overline{f} \in \FF_q[x,y]$. Or it could mean that we are considering the affine/projective curve
defined by a given more general system of equations over $\FF_q$. In all cases we will ignore the cost
of computing the genus $g$ of $\overline{C}$.
Moreover, in case $g = 0$ 
we assume that it is easy to realize $\overline{C}$ as a plane conic (using the anticanonical embedding) and if $g = 1$ we ignore
the cost of finding a plane Weierstrass model. By the Hasse-Weil bound every genus one curve over $\FF_q$ is elliptic, so this
is indeed possible.
If $g \geq 2$ then we assume that one can easily decide whether $\overline{C}$ is hyperelliptic or not (note
that over finite fields, curves are hyperelliptic iff they are geometrically hyperelliptic, so there is no ambiguity here). If it is then
we suppose that it is easy to find a generalized Weierstrass model. If not then it is assumed that one
can effectively compute a canonical embedding
\[ \kappa : \overline{C} \hookrightarrow \PPq^{g-1} \]
along with a minimal set of generators for the ideal of its image. The latter will usually be our starting point. 
Most of the foregoing tasks are tantamount to computing certain Riemann-Roch spaces.
There is extensive literature on this functionality, which 
has been implemented in several computer algebra packages, such as Magma~\cite{magma} and Macaulay2~\cite{macaulay}. 

The idea is then to use the output polynomial $f$ as input to a recent algorithm due to the second author~\cite{tuitman1,tuitman2} for computing
the Hasse-Weil zeta function of $\overline{C}$. This algorithm uses $p$-adic cohomology, which it represents through the map 
$\pi : C \rightarrow \PPK^1 : (x,y) \mapsto x$. The algorithm only works if $C$ and $\pi$ have appropriate reduction modulo $p$, in a rather
subtle sense for the precise description of which we refer to \cite[Ass.\,1]{tuitman2}. This condition is needed to be able to apply a
comparison theorem between the (relative) $p$-adic cohomology of $\overline{C}$ and
the (relative) de Rham cohomology of $C \otimes \QQ_q$, which is where the actual computations are done. 
For such a theorem to hold, by dimension arguments it is necessary that $C$ and $\overline{C}$ have the same genus, whence our condition (ii).
This may be insufficient, in which case $f$ will be rejected, but for $p > 2$ our experiments show that this
is rarely a concern as soon as $q$ is sufficiently large. Moreover, in many cases below, our construction leaves enough freedom to retry in the event of a failure.

The algorithm from~\cite{tuitman1,tuitman2} has a running time that is sextic in $\deg \pi$, which equals the degree in $y$ of $f$, so it is important to keep this value within reason.
Because the $\FF_q$-gonality of $\overline{C}$ is an innate lower bound, it is natural to try to meet this value, whence our condition (iii). 
At the benefit of other parameters affecting the complexity, one could imagine it being useful to allow input polynomials whose degree in $y$ exceeds the $\FF_q$-gonality
of $\overline{C}$, but in all cases that we studied the best performance results were indeed obtained using a gonality-preserving lift. 
At the same time, looking for such a lift is a theoretically neat problem. 

\begin{remark}
For the purpose of point counting, it is natural to wonder why we lift to $\mathcal{O}_K$, and not to the ring $\ZZ_q$, which is a priori easier. In fact, most computations in the algorithm from~\cite{tuitman1,tuitman2} are carried out to some finite $p$-adic precision~$N$, so 
it would even be sufficient to lift to $\mathcal{O}_K/(p^N) = \ZZ_q/(p^N)$. 
 A first reason for lifting to $\mathcal{O}_K$ is simply that this turns out to be possible in the cases that we studied, without additional difficulties.
A second more practical reason is that
at the start of the 
algorithm from~\cite{tuitman1,tuitman2} some integral bases have to be computed in the function field of the curve. Over a number field $K$ this is standard
and implemented in Magma, but to finite $p$-adic precision it is not clear how to do this, and in particular no implementation is available. Therefore, the integral bases
are currently computed to exact precision, and we need $f$ to be defined over $\mathcal{O}_K$.
\end{remark}
  

\paragraph*{Contributions} 
As explained in Section~\ref{section_preliminary} the cases
where $\overline{C}$ is rational, elliptic or hyperelliptic are straightforward. In this article we give a recipe for tackling Problem~\ref{liftingproblem} in the case of 
curves of $\FF_q$-gonality $3$ and $4$. 
Because of their practical relevance, our focus lies on curves having genus at most five, which is large
enough for the main trigonal and tetragonal phenomena to be present. 
The details can be found in Section~\ref{section_lowgenus};
more precisely in Sections~\ref{section_genus3},~\ref{section_genus4} and~\ref{section_genus5}
we attack Problem~\ref{liftingproblem} for curves of genus three, four and five, respectively,
where we restrict ourselves
to finite fields $\FF_q$ having odd characteristic. Each of these
sections is organized in a stand-alone way, as follows:
\begin{itemize}
  \item In a first part we classify curves by their $\FF_q$-gonality $\gamma$ and solve Problem~\ref{liftingproblem} in its basic 
  version (except for some pathological cases such as pentagonal curves in genus four or hexagonal curves in genus five, which
  are irrelevant for point counting because these can only exist over extremely small fields).
   If the reader is interested in such a basic solution only, he/she can skip the other parts, which are more technical. 
  \item Next, in an optimization part we take into account the fact that the actual input to the algorithm from~\cite{tuitman1,tuitman2} must be monic when considered as a polynomial in $y$.
  This is easily achieved:
if we write
\[ f = f_0(x)y^\gamma + f_1(x)y^{\gamma - 1} + \dots + f_{\gamma - 1}(x)y + f_\gamma(x), \]
then the birational transformation $y \leftarrow y / f_0(x)$ gives
\begin{equation} \label{mademonic} 
  y^\gamma + f_1(x)y^{\gamma - 1} + \dots + f_{\gamma - 1}(x)f_0(x)^{\gamma - 2}y + f_\gamma(x)f_0(x)^{\gamma - 1},
\end{equation}
which still satisfies (i), (ii) and (iii). But one sees that the degree in $x$ inflates, and this affects the running time.
We discuss how our basic solution to Problem~\ref{liftingproblem} can be enhanced such that (\ref{mademonic}) becomes a more compact expression.
\item  We have implemented the algorithms from this paper in the computer algebra system Magma. The resulting package is
called \verb{goodmodels{ and can be found at the webpage \url{http://perswww.kuleuven.be/jan_tuitman}. In a third part we report on this 
implementation and on how it performs in composition with the algorithm from~\cite{tuitman1,tuitman2} for computing Hasse-Weil zeta functions. 
We give concrete runtimes, memory usage and failure rates, but avoid a detailed complexity analysis, because in any case the lifting step is heavily dominated 
by the point counting step. All computations were carried out with Magma v2.22 on a single Intel Core i7-3770 CPU running at 3.40 GHz. The code 
used to generate the tables with running times, memory usage and failure rates can be found in the subdirectory \verb{./profiling{ of \verb{goodmodels{.
\end{itemize}
As we will see, the case of trigonal curves of genus five provides a natural transition to 
the study of general curves of $\FF_q$-gonality $3$ and $4$. These are discussed in
Section~\ref{section_lowgonality}, albeit in a more sketchy way.

\paragraph*{Consequences}
 The main consequences of our work are that 
\begin{itemize}
  \item computing Hasse-Weil zeta functions using $p$-adic cohomology
has now become practical
on virtually all curves of genus at most five over finite fields $\FF_q$ of (small) odd characteristic, 
  \item the same conclusion for curves of $\FF_q$-gonality at most four looms around the corner, even though some hurdles remain, 
  as explained in Section~\ref{section_lowgonality},
  \item we have a better understanding of which $\FF_q$-gonalities can occur for curves of genus at most five, see the end of Section~\ref{section_firstfacts} for
  a summarizing table.
\end{itemize}
We stress that the general genus five curve, let alone the general tetragonal curve of any given genus,
cannot be tackled using any of the previous Kedlaya-style point counting algorithms, that were designed
to deal with elliptic curves~\cite{satoh}, hyperelliptic curves~\cite{denefvercauteren,harrison,harveylarger,hubdeform,kedlaya}, superelliptic curves~\cite{gaudrygurel,minzlaff}, $C_{ab}$ curves~\cite{CHV,DVCab,walker} and nondegenerate curves~\cite{CDV,tuitmanthesis},
in increasing order of generality. We refer to~\cite{CV} for
a discussion of which classes of curves do admit a nondegenerate model.


\paragraph*{A reference problem ($\dagger$)} At sporadic places in this article, we refer to a paper that develops
its theory over $\CC$ only, while in fact we need it over other fields, such as $\overline{\FF}_q$. 
This concern mainly applies to the theory of genus five curves due to Arbarello, Cornalba, Griffiths and Harris~\cite[VI.\S4.F]{cornalba}. We are convinced that most of the time
this is not an issue (the more because we rule out even characteristic) but we did not sift every one of these references to the bottom to double-check this: we content ourselves with the fact that things work well in practice.  
In our concluding
Section~\ref{section_lowgonality} on trigonal and tetragonal curves, the field characteristic becomes a more serious issue, for instance in the Lie algebra
method developed by de Graaf, Harrison, P\'ilnikov\'a and Schicho~\cite{GHPS}. More comments on this will be given there.
Each time we cite a $\CC$-only (or characteristic zero only) reference whose statement(s) we carry over
to finite characteristic without having verified the details, we will
indicate this using the dagger symbol $\dagger$.

\paragraph*{Acknowledgements} We would like to thank Arnaud Beauville, Tom De Medts, Jeroen Demeyer, Steve Donnelly and Josef Schicho for answering several of our
questions. A large part of this paper was prepared while the first author was affiliated with the University of Ghent. 
The second author is a postdoctoral research fellow of the Research Foundation Flanders (FWO). Further support for this research was received from the project G093913N of the Research Foundation Flanders (FWO) and
from
the European Commission through the European Research
Council under the FP7/2007-2013 programme with ERC Grant Agreement 615722 MOTMELSUM.

\section{Background} \label{section_preliminary}

\subsection{First facts on the gonality} \label{section_firstfacts}

Let $k$ be a field and let $C$ be a curve over $k$.
The geometric gonality $\gamma_\text{geom}$ of $C$ is a classical invariant. It is $1$ if and only if the genus of $C$ equals $g = 0$, while for curves
of genus $g \geq 1$, by Brill-Noether theory $\gamma_\text{geom}$ lies in the range 
\[ 2, \dots, \lceil g / 2 \rceil + 1. \]
For a generic curve the upper bound $\lceil g / 2 \rceil + 1$ is met~\cite{CDPR}, but in fact 
each of the foregoing values can occur: inside the moduli space of curves of genus $g \geq 2$ the corresponding locus has dimension $\min \{ 2g - 5 + 2\gamma_\text{geom}, 3g - 3 \}$; see~\cite[\S8]{arbarello}${}^\dagger$.
From a practical point of view, determining the geometric gonality of a given curve is usually a non-trivial computational task, although in theory it can be 
computed using so-called scrollar syzygies~\cite{weimann}.

In the arithmetic (= non-geometric) case the gonality has seen much less study, even for classical fields
such as the reals~\cite{coppensmartens}. 
Of course $\gamma_\text{geom}$ is always less than or equal to the $k$-gonality $\gamma$, but the inequality may be strict.
In particular the Brill-Noether upper bound $\lceil g / 2 \rceil + 1$ is no longer valid. 
For curves of genus $g = 1$ over certain fields $\gamma$ can even be arbitrarily large~\cite{clark}.
As for the other genera, using the canonical or anticanonical linear system one finds 
\begin{itemize}
  \item if $g = 0$ then $\gamma \leq 2$,
  \item if $g \geq 2$ then $\gamma \leq 2g-2$.
\end{itemize}
These bounds can be met. We refer to~\cite[Prop.\,1.1]{poonengonality} and the references therein for
precise statements, along with some additional first facts.

If $k = K$ is a number field 
then the notion of $K$-gonality has enjoyed more 
attention, both from
a computational~\cite{derickx,derickxvanhoeij} and a theoretical~\cite{poonengonality} point of view, especially
in the case where $C$ is a modular curve. This is due
to potential applications towards effective versions of the uniform boundedness conjecture; see~\cite{sutherland} for an overview. In the non-modular case not much literature
seems available, but our rash guess would be that almost all (in any honest sense) curves of genus $g \geq 2$ over $K$ meet the upper bound $\gamma \leq 2g - 2$. This is distantly supported by the Franchetta conjecture; see again~\cite[Prop.\,1.1]{poonengonality} and the references therein for a more extended discussion.

Over finite fields $k = \FF_q$ the notion has
attracted the attention of coding theorists in the context of Goppa codes~\cite{tsfasman}. They proved the following result:

\begin{lemma}
  If the $\overline{C}$ is a curve over a finite field $\FF_q$ then its $\FF_q$-gonality is at most $g + 1$. Moreover, if
  equality holds then $g \leq 10$ and $q \leq 31$.
\end{lemma}

\begin{proof} See \cite[\S4.2]{tsfasman}. \end{proof}

In \cite[\S4.2]{tsfasman} it is stated as an open problem to find tighter bounds for the $\FF_q$-gonality. 
In fact we expect the sharpest possible upper bound to be $\lceil g/2 \rceil + 1 + \varepsilon$ for some small $\varepsilon$; maybe $\varepsilon \leq 1$ is sufficient
as soon as $q$ is large enough. 
A byproduct of this paper is a better understanding of which $\FF_q$-gonalities can occur for curves of genus at most five, 
in the cases where $q$ is odd (the cases where $q$ is even should be analyzable in a similar way). The following table summarizes this. 
\begin{center}
\small
\begin{tabular}{c|c|c|c|c}
  $g$ & Brill-Noether & possible $\FF_q$-gonalities & possible $\FF_q$-gonalities & $B$ \\
         &  upper bound &  (union over all odd $q$)            & (for a given odd $q > B$)             &        \\
  \hline
  $0$ & $1$ & $1$ & 1 & 1 \\
  $1$ & $2$ & $2$ & 2 & 1 \\
  $2$ & $2$ & $2$ & 2 & 1 \\
  $3$ & $3$ & $2,3,4$ & $2,3$ & $29$\\
  $4$ & $3$ & $2,3,4,5$ & $2,3,4$ & $7$\\
  $5$ & $4$ & $2,3,4,5,6^?$ & $2,3,4,5$ & $3$\\
\end{tabular}
\end{center}
For background we refer to Section~\ref{section_preliminary_comments} (for $g \leq 2$), Lemma~\ref{genus3gonality} (for $g=3$), Lemma~\ref{genus4gonality} (for $g=4$), and Lemma~\ref{genus5gonality}, Remark~\ref{remarkjeroen} and Remark~\ref{remarkgon6} (for $g=5$).
The question mark indicates that over $\FF_3$ there might exist curves of genus $g=5$ having $\FF_3$-gonality $6$, but there also might not exist such curves, see
Remark~\ref{remarkgon6}.

%

\subsection{Baker's bound} \label{section_bakersbound}

Throughout a large part of this paper we will use the convenient language of Newton polygons.
Let \hfill \phantom{x}
\begin{wrapfigure}{r}{3.1cm} 
  \hfill \includegraphics[width=3cm]{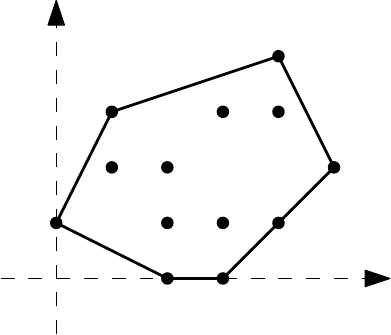} 
\end{wrapfigure}  
\[ f = \sum_{(i,j) \in \ZZ_{\geq 0}^2} c_{i,j} x^iy^j \in k[x, y] \] 
be an irreducible polynomial over a field $k$. 
Then its Newton polygon
$ \Delta(f)$ is defined as
$\conv \left\{ \, \left. (i,j) \in \ZZ_{\geq 0}^2 \, \right| \, c_{i,j} \neq 0 \,  \right\} \subset \RR^2 $. Note that $\Delta(f)$ lies in the first quadrant and meets 
the coordinate axes in at least one point each, by the irreducibility of $f$. 
Let $C$ be the affine curve that is cut out by $f$. 
Then one has the following bounds on the genus and the gonality of $C$, purely in terms of the combinatorics of $\Delta(f)$.

\paragraph*{Genus} The genus of $C$ is at most \emph{the number of points in the interior} of $\Delta(f)$ having integer coordinates:
  this is Baker's theorem. See~\cite[Thm.\,2.4]{beelen} for an elementary proof and~\cite[\S10.5]{coxlittleschenck} for a more conceptual version
  (using adjunction theory on toric surfaces).  
  If one fixes the Newton polygon then Baker's bound on the genus is generically attained, i.e.\ meeting the bound is a 
non-empty Zariski-open condition; this result is essentially due to Khovanskii \cite{khovanskii}.
An explicit sufficient generic condition is that $f$ is nondegenerate with respect to its Newton polygon~\cite[Prop.\,2.3, Cor.\,2.8]{CDV}.

\paragraph*{Gonality} The $k$-gonality is at most the \emph{lattice width} $\lwidth(\Delta(f))$ of $\Delta(f)$.
  By definition, the lattice width is the minimal height $d$ of a horizontal strip
  \[ \left\{ \left. \, (a,b) \in \RR^2 \, \right| \, 0 \leq b \leq d \, \right\} \]
  inside which $\Delta(f)$ can be mapped using a unimodular transformation, i.e.\ an affine transformation of $\RR^2$ with linear
  part in $\GL_2(\ZZ)$ and translation part in $\ZZ^2$. 
  \begin{center}
  \includegraphics[width=9cm]{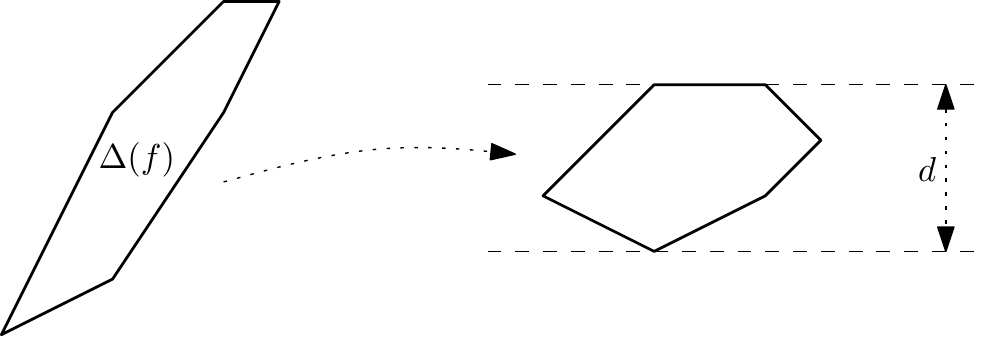}
  \end{center}
  This is discussed in \cite[\S2]{caco}, but briefly the argument goes as follows.
  By applying the same transformation to the exponents, which is a $k$-rational birational change of variables, our
  polynomial $f$ can be transformed along with its Newton polygon. When orienting $f$ in this way one 
  obtains $\deg_y f = \lwidth(\Delta(f))$, and
  the gonality bound follows by considering the $k$-rational map $(x,y) \mapsto x$. 
 If a unimodular transformation can be used to
  transform $\Delta(f)$ into
  \begin{center}
  \polfig{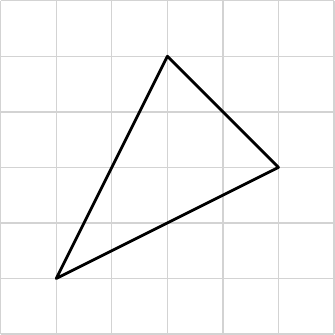}{2.4}{2.4}{$2 \Upsilon$} 
  \begin{minipage}[b]{1cm}
  \begin{center}
  or 
  \end{center}
  \vspace{0.65cm}
  \end{minipage}
  \polfig{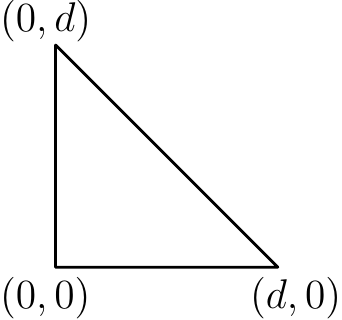}{2.4}{2.4}{$d \Sigma$} 
  \end{center}
  for $d \geq 2$, then the \emph{geometric} gonality enjoys the sharper bound $\lwidth(\Delta(f)) - 1$ (amounting to $3$ resp.\ $d-1$); see \cite[Thm.\,3]{caco}. 
  If one fixes the Newton polygon then the sharpest applicable foregoing upper bound on the geometric gonality, i.e.\ 
  \begin{itemize}
    \item $\lwidth(\Delta(f)) - 1$ in the exceptional cases $2\Upsilon$, $d\Sigma$ ($d \geq 2$),
    \item $\lwidth(\Delta(f))$ in the non-exceptional cases,
  \end{itemize}
  is generically met, and again
  nondegeneracy is a sufficient condition~\cite[Cor.\,6.2]{linearpencils}. In fact, the slightly weaker condition of meeting Baker's genus bound
  is already sufficient \cite[\S4]{linearpencils}.
  
  \begin{remark} The results from~\cite{linearpencils} 
  are presented in characteristic zero only, but~\cite[Cor.\,6.2]{linearpencils} holds in finite characteristic too, as can be seen as follows. Assume for simplicity that $\Delta(f)$
  is not of the form $2\Upsilon$ or $d\Sigma$ for some $d \geq 2$, these cases are easy to deal with separately. Suppose that $C$ meets Baker's bound but that the gonality of $C$ is strictly less than $\lwidth(\Delta(f))$, say
  realized by a map $\pi : C \rightarrow \PPq^1$.
  We split this map in the usual way into a purely inseparable and a separable part 
  \[ C \stackrel{F_q}{\longrightarrow} C^{F_q} \stackrel{\pi_s}{\longrightarrow} \PPq^1, \]
  where $F_q$ denotes an appropriate Frobenius power and $C^{F_q}$ is the curve defined by $f^{F_q}$, the polynomial obtained by applying $F_q$ to each coefficient of $f$. Note that $\Delta(f) = \Delta(f^{F_q})$, so one sees that $C^{F_q}$ also meets Baker's bound because Frobenius preserves the genus~\cite[Prop.\,IV.2.5]{hartshorne}. Clearly $\deg \pi_s < \lwidth(\Delta(f^{F_q}))$. Now the crucial ingredient in the proof of~\cite[Cor.\,6.2]{linearpencils} is a theorem due to Serrano on the possibility of extending morphisms from curves to ambient surfaces,
which assumes $\chr k = 0$. However as Serrano points out~\cite[Rmk.\,3.12]{serrano} his theorem also holds in finite characteristic, provided that the morphism is separable, the ambient surface $S$ is rational, and $h^0(\mathcal{O}_S(C))$ is large enough compared to the degree of the morphism to be extended. The reader can verify that these conditions are satisfied when applying the proof of~\cite[Thm.\,6.1]{linearpencils} to $\pi_s$, leading to the conclusion that it is necessarily of the form $(x,y) \mapsto x^ay^b$ for some pair of coprime integers $a,b$. This contradicts that $\deg \pi_s < \lwidth(\Delta(f^{F_q}))$.
  \end{remark}
  
  Summing up in the non-geometric case, if we are not in the exceptional cases $2 \Upsilon, d\Sigma$ ($d \geq 2$) then meeting Baker's bound
  is sufficient for the $k$-gonality to equal $\lwidth(\Delta(f))$. In the exceptional cases the $k$-gonality is either $\lwidth(\Delta(f))$ or $\lwidth(\Delta(f)) - 1$.\\

\noindent This yields 
a large class of defining polynomials $\overline{f} \in \FF_q[x,y]$ for which finding an 
$f \in \mathcal{O}_K[x,y]$ satisfying (i), (ii) and (iii) is easy. Indeed, by semi-continuity the
genus cannot increase under reduction modulo $p$. Therefore
if $\overline{f}$ attains Baker's upper bound on the genus, then it suffices to pick
 any $f \in \mathcal{O}_K[x,y]$ that reduces to $\overline{f}$ mod $p$, in such a way that $\Delta(f) = \Delta(\overline{f})$:
 the corresponding curve $C / K$ necessarily attains Baker's upper bound, too.
If moreover we are not in the exceptional cases $2\Upsilon$ and $d\Sigma$ ($d \geq 2$), then from the foregoing discussion we
know that both the $\FF_q$-gonality of $\overline{C}$ and the $K$-gonality of $C$ are equal to
$\lwidth(\Delta(\overline{f})) = \lwidth(\Delta(f))$. A unimodular transformation then ensures that $\deg_y f =
\lwidth(\Delta(f))$ as desired; 
such a transformation is computationally easy to find~\cite{feschet}.

It is therefore justifiable to say
that conditions (i), (ii) and (iii) are easy to deal with for almost all polynomials $\overline{f} \in \FF_q[x,y]$.
But be cautious: this does not mean that almost all \emph{curves} $\overline{C} / \FF_q$ are defined by such a polynomial. In terms of moduli, the locus of curves for which this is true has dimension $2g+1$, except if $g=7$ where it is $16$; see \cite[Thm.\,12.1]{CV}. Recall that the moduli space of curves of genus $g$ has dimension $3g - 3$, so as soon as $g \geq 5$ the defining polynomial $\overline{f}$ of a plane model of a generic curve $\overline{C}/ \FF_q$ of genus $g$ can never attain Baker's bound.
For such curves, the foregoing discussion becomes
\emph{counterproductive}: if we take a naive coefficient-wise lift $f \in \mathcal{O}_K[x,y]$ of $\overline{f}$, then it is very likely to satisfy Baker's bound,
causing an increase of genus. This shows that $f$ has to be constructed with more care, which
is somehow the main point of this article.

\subsection{Preliminary discussion} \label{section_preliminary_comments}

We will attack Problem~\ref{liftingproblem} in the cases where the genus $g$ of $\overline{C}$ is at most five (in Section~\ref{section_lowgenus}) or
the $\FF_q$-gonality $\gamma$ of $\overline{C}$ is at most four (in Section~\ref{section_lowgonality}), where we recall
our overall assumption that $q$ is odd. In this section we quickly discuss the cases where $g$ and/or $\gamma$ are at most $2$.

\begin{remark}
Note that for the purpose of computing the Hasse-Weil zeta function using
the algorithm from~\cite{tuitman1,tuitman2}, 
the characteristic $p$ of $\FF_q$ should moreover not be too large: this restriction is common 
to all $p$-adic point counting algorithms. For the lifting methods described in the current paper,
the size of $p$ does not play a role.
\end{remark}



If $\overline{C}$ is a curve of genus $g = 0$ then 
we can assume that $\overline{C} = \mathbb{P}^1$, because every plane conic carries at least one $\FF_q$-point, and projection from that point 
gives an isomorphism to the line. In particular $\gamma = 1$ if and only if $g = 0$, in which case 
Problem~\ref{liftingproblem} can be addressed by simply outputting $f = y$.

Next, if $g = 1$ then we can assume that $\overline{C}$ is defined by a polynomial $\overline{f} \in \FF_q[x,y]$ in Weierstrass form, i.e.\ $\overline{f} = y^2 - \overline{h}(x)$ for some squarefree cubic $\overline{h}(x) \in \FF_q[x]$. In this case $\gamma = 2$, and any $f \in \mathcal{O}_K[x,y]$
for which $\Delta(f) = \Delta(\overline{f})$ will address Problem~\ref{liftingproblem} (for instance because Baker's bound is attained, or because a non-zero discriminant must lift to a non-zero discriminant).

Finally, 
if $g \geq 2$ then $\overline{C}$ is geometrically hyperelliptic if and only if $\kappa$ realizes $\overline{C}$ as a degree $2$ cover
of a curve of genus zero \cite[IV.5.2-3]{hartshorne}. By the foregoing discussion the latter is isomorphic to $\PPq^1$, and therefore every geometrically hyperelliptic curve $\overline{C} / \FF_q$ admits an $\FF_q$-rational degree $2$ map to $\PPq^1$. In particular, one can unambiguously talk about hyperelliptic curves over $\FF_q$. In this case it is standard how to produce a defining polynomial $\overline{f} \in \FF_q[x,y]$
that is in Weierstrass form, i.e.\ $\overline{f} = y^2 - \overline{h}(x)$ for some squarefree $\overline{h}(x) \in \FF_q[x]$. Then again any $f \in \mathcal{O}_K[x,y]$
for which $\Delta(f) = \Delta(\overline{f})$ will address Problem~\ref{liftingproblem}.

\begin{remark} \label{remark_gonalityoverFq}
Let $g^1_d$ be a complete base-point free $\FF_q$-rational linear pencil of degree~$d$ on a non-singular projective curve $\overline{C} / \FF_q$. Then
from standard arguments in Galois cohomology (that are specific to finite fields) it follows that this $g^1_d$
automatically contains an $\FF_q$-rational effective divisor, which can 
be used to construct 
an $\FF_q$-rational map to $\PPq^1$ of degree $d$. See for instance the proof of~\cite[Lem.\,6.5.3]{gilles}. 
This gives another way of seeing that a geometrically hyperelliptic curve over $\FF_q$ is automatically $\FF_q$-hyperelliptic,
because the hyperelliptic pencil $g^1_2$ is unique, hence indeed defined over $\FF_q$. 
The advantage of this argument is that it is more flexible: for instance it
also shows that a geometrically trigonal curve $\overline{C} / \FF_q$ of genus $g \geq 5$ always
admits an $\FF_q$-rational degree $3$ map to $\PPq^1$, again because 
the $g^1_3$ on such a curve is unique. So we can unambiguously talk about trigonal curves from genus five on.
\end{remark}

Summing up, throughout the paper, it suffices to consider curves of $\FF_q$-gonality $\gamma > 2$, so that the canonical map $\kappa: \overline{C} \rightarrow \PPq^{g-1}$ is an embedding. In particular we have $g \geq 3$.
From the $p$-adic point counting viewpoint, all omitted cases are covered by the algorithms of Satoh~\cite{satoh} and Kedlaya~\cite{harrison,kedlaya}.

\section{Curves of low genus} \label{section_lowgenus}

\subsection{Curves of genus three} \label{section_genus3}

\subsubsection{Lifting curves of genus three} \label{basicsolutiontogenus3}

Solving Problem~\ref{liftingproblem} in genus three in its basic version is not hard, so we consider this as a warm-up discussion. We first
analyze which $\FF_q$-gonalities can occur:

\begin{lemma} \label{genus3gonality}
Let $\overline{C} / \FF_q$ be a non-hyperelliptic curve of genus $3$ and 
$\FF_q$-gonality $\gamma$, and assume that $q$ is odd. If $\# \overline{C}(\FF_q) = 0$ then $\gamma = 4$, while if $\# \overline{C}(\FF_q) > 0$ (which is guaranteed if $q > 29$) then $\gamma = 3$.

\end{lemma}

\begin{proof}
  Using the canonical embedding we can assume that $\overline{C}$ is a smooth plane quartic. It is classical that 
  such curves have geometric gonality $3$, and that each gonal map arises
  as projection from a point on the curve. For a proof see \cite[Prop.\,3.13]{serrano}, where things are formulated in characteristic zero, but the same argument works in positive characteristic; alternatively one can consult~\cite{homma}. In particular if there is no $\FF_q$-point then there is no rational gonal map and $\gamma > 3$.
  But then a degree $4$ map can be found by projection from an $\FF_q$-point outside the curve. By \cite[Thm.\,3(2)]{howe}
  there exist pointless non-hyperelliptic curves of genus three over $\FF_q$ if and only if $q \leq 23$ or $q = 29$.
\end{proof}


We can now address Problem~\ref{liftingproblem} as follows. As in the proof we assume that $\overline{C}$ is given as a smooth 
quartic in $\PPq^2$.
First suppose that $\# \overline{C}(\FF_q) = 0$. Because this is possible for $q \leq 29$ only, the occurrence of this event can
be verified exhaustively. In this case
the Newton polygon
of the defining polynomial
$\overline{f} \in \FF_q[x,y]$ 
  of the affine part of $\overline{C}$ equals:
\begin{center}
\polfig{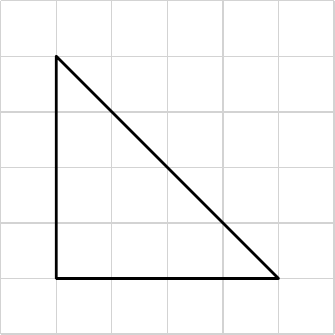}{2.4}{2.4}{$\Delta_3^{0,0}$}
\end{center}
In particular Baker's bound is attained, and a naive Newton polygon preserving lift 
$f \in \mathcal{O}_K[x,y]$ automatically addresses (i), (ii) and (iii). If $\#\overline{C}(\FF_q)>0$ then one
picks a random $\FF_q$-point $P$ (which can be found quickly) and one applies
a projective transformation that maps $P$ to $(0 : 1 : 0)$.
After doing so the Newton polygon of $\overline{f} \in \FF_q[x,y]$ becomes contained in (and typically equals):
\begin{center}
\polfig{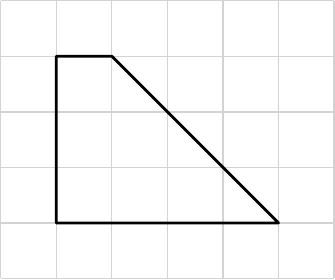}{2.4}{2}{$\Delta_3^{1,0}$}
\end{center}
Again Baker's bound is attained, and a naive Newton polygon preserving lift 
$f \in \mathcal{O}_K[x,y]$ satisfies (i), (ii) and (iii).

It is important to transform the curve \emph{before} lifting to characteristic $0$. 
Indeed, if one would immediately lift our input quartic to a curve $C \subset \PPK^2$ then it is highly likely that $C(K) = \emptyset$, and therefore
that the $K$-gonality equals $4$ (by the same proof as above). This type of reasoning plays an important role throughout the paper, often in a more subtle way than here.

\begin{remark}[purely notational]
The indices $i,j$ in $\Delta_3^{i,j}$ refer to the multiplicities of intersection of $\overline{C}$ with the line at infinity at
the coordinate points $(0:1:0)$ and $(1:0:0)$, assuming that it is defined by a polynomial having Newton polygon 
$\Delta_3^{i,j}$. Note that $\Delta_3^{0,0}$ is just another way of writing $3\Sigma$.
\end{remark}

\noindent \varhrulefill[0.4mm]
\vspace{-0.3cm}

\begin{algo} \label{algorithm_genus3}
Lifting curves of genus $3$: basic solution 

\vspace{-0.2cm}
\noindent \varhrulefill[0.4mm]

\noindent \textbf{Input:} non-hyperelliptic genus $3$ curve $\overline{C}$ over $\FF_q$

\noindent \textbf{Output:} lift $f \in \mathcal{O}_K[x,y]$ satisfying (i), (ii), (iii) that is supported

\noindent \qquad \qquad \qquad $\bullet$ on $\Delta_3^{0,0}$ if $\overline{C}(\FF_q) = \emptyset$, or else

\noindent \qquad \qquad \qquad $\bullet$ on $\Delta_3^{2,0}$ 

\vspace{-0.2cm}
\noindent \varhrulefill[0.4mm]

\noindent \small 1 \normalsize: $\overline{C} \gets \text{CanonicalImage}(\overline{C})$ in $\PPq^2 = \proj \FF_q[X,Y,Z]$

\noindent \small 2 \normalsize: \textbf{if} $q > 29$ or $\overline{C}(\FF_q) \neq \emptyset$ (verified exhaustively) \textbf{then}

\noindent \small 3 \normalsize: \qquad $P := \text{Random}(\overline{C}(\FF_q))$

\noindent \small 4 \normalsize: \qquad apply automorphism of $\PPq^2$ transforming $T_P(\overline{C})$ into $Z=0$

\noindent \small 5 \normalsize: \qquad \textcolor{white}{apply automorphism of $\PPq^2$ transforming} and $P$ into $(0:1:0)$

\noindent \small 6 \normalsize: \textbf{return} NaiveLift(Dehomogenization${}_Z$(DefiningPolynomial($\overline{C}$)))

\vspace{-0.2cm}
\noindent \varhrulefill[0.4mm]
\end{algo}

\subsubsection{Optimizations} \label{optim_genus3}

For point counting purposes we can of course assume that $q > 29$, so that $\gamma = 3$. By applying (\ref{mademonic}) to 
a polynomial with Newton polygon $\Delta_3^{1,0}$
one ends up with a polynomial that is monic in $y$ and that has degree $4 + (\gamma - 1) = 6$ in $x$. This can be improved: 
in addition to mapping $P$ to $(0:1:0)$, we can have its tangent line $T_P(\overline{C})$
sent to the line at infinity. If we then lift $\overline{f}$ to $\mathcal{O}_K[x,y]$ we find an $f$ whose Newton polygon is contained in (and typically equals):
\begin{center}
\polfig{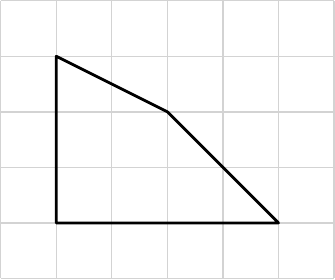}{2.4}{2}{$\Delta_3^\text{2,0}$}
\end{center}
In particular $f$ is monic (up to a scalar) and $\deg_x f \leq 4$. We can in fact achieve
  $\deg_x f = 3$ in all cases of practical interest. Indeed, with an asymptotic chance of $1/2$ our tangent line $T_P(\overline{C})$ intersects $\overline{C}$ 
  in two other rational points. The above construction leaves enough freedom to position one of those points $Q$ at $(1:0:0)$. The resulting lift $f$ then becomes contained in (and typically equals)
  \begin{center}
\polfig{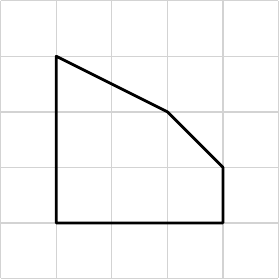}{2}{2}{$\Delta_3^\text{2,1}$}
\end{center}
In the case of failure we retry with another $P$. 
If $q > 59$ (say) then there are enough $\FF_q$-points $P \in \overline{C}$ for this approach to work with near certainty, although
there might exist sporadic counterexamples well beyond that point.

%

\begin{remark}[non-generic optimizations] \label{genus3flexremark}
For large values of $q$ one might want to pursue a further compactification of the Newton polygon. Namely, if one manages to
choose $P \in \overline{C}(\FF_q)$ such that it is an ordinary flex or such that $T_P(\overline{C})$ is a bitangent, 
then $T_P(\overline{C})$ meets $\overline{C}$ in a unique other point $Q$, which is necessarily defined over $\FF_q$.
By proceeding as before one respectively ends up inside the first and second polygon below.
If one manages to let $P \in \overline{C}(\FF_q)$ be a non-ordinary flex, i.e.\ a hyperflex, then positioning it at $(0:1:0)$ results in a polygon of the third form:
     \begin{center}
\polfig{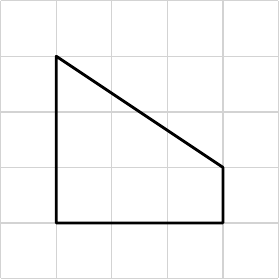}{2}{2}{$\Delta_3^\text{3,1}$} \hspace{-1cm}
\polfig{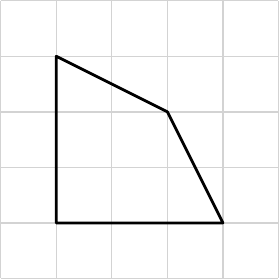}{2}{2}{$\Delta_3^\text{2,2}$} \hspace{-1cm}
\polfig{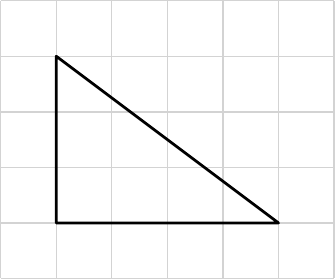}{2.4}{2}{$\Delta_3^\text{4,0}$} 
  \end{center}
Heuristically, as $q \rightarrow \infty$ we expect to be able to realize the first two polygons with probablities $1-1/e$ and $1 - 1/\sqrt{e}$, respectively; more background can be found in an \texttt{arXiv} version of our paper (\texttt{1605.02162v2}). In contrast
the hyperflex case $\Delta_3^\text{4,0}$ is very exceptional, but we included it in the discussion because it corresponds to the well-known class
  of $C_{3,4}$ curves: even though $\deg_x f = 4$ here, the corresponding point count is slightly faster.
\end{remark}

\subsubsection{Implementation} \label{section_genus3implementationandtimings}

We now report on timings, memory usage and failure rates of our implementation of the algorithms in this section for various values of $p$ and $q=p^n$. 
The first column in each table contains the time used to 
compute the lift to characteristic~$0$ averaged over $1000$ random examples. Then the second column gives the time used by the point 
counting code \verb{pcc{ from \cite{tuitman1,tuitman2} averaged over $10$ different random examples. Next, the third column contains the total 
memory used in the computation. Finally, the last column gives the number of examples out of the $1000$ where we did not find a lift satisfying 
\cite[Ass.\,1]{tuitman2}, which each time turned out to be $0$, i.e.\ we always found a good lift.

\bigskip

\noindent \scriptsize
\tabcolsep=0.11cm
\begin{tabular}{r||r|r|r|r}
             & time     &time    & space  & fails    \\
$p$          & lift(s)  &pcc(s)  & (Mb)   & /1000    \\
\hline \hline
$11$         & 0.2      &$0.2$   & $ 32$ &        0  \\
$67$         & 0.2      &$0.6$   & $ 32$ &        0  \\
$521$        & 0.2      &$4.2$   & $ 64$ &        0  \\
$4099$       & 0.2      & $41$   & $165$ &        0  \\
$32771$      & 0.2      &$590$   & $1124$ &        0 
\end{tabular}
\quad
\begin{tabular}{r||r|r|r|r}
             & time     &time    & space  & fails  \\
$q$          & lift(s)  &pcc(s)  & (Mb)   & /1000  \\
\hline \hline
$3^5$        & $0.4$    &$2.4$   & $ 64$  &     0  \\
$7^5$        & $0.4$    &$6.6$   & $ 64$  &     0  \\
$17^5$       & $0.4$    &$12$    & $ 76$  &     0  \\
$37^5$       & $0.4$    &$26$    & $124$  &     0  \\
$79^5$       & $0.4$    &$66$    & $241$  &     0 
\end{tabular}
\quad
\begin{tabular}{r||r|r|r|r}
                &time   &time     & space  & fails  \\
$q$             &lift(s)&pcc(s)   & (Mb)   & /1000  \\
\hline \hline
$3^{10}$        &$0.5$ &$15$     & $76$   &     0  \\
$7^{10}$        &$0.6$ &$40$     &$118$   &     0  \\
$17^{10}$       &$0.7$ &$82$     &$241$   &     0  \\
$37^{10}$       &$0.7$ &$181$    &$403$   &     0  \\
$79^{10}$       &$0.8$ &$473$    &$831$   &     0 
\end{tabular}

\bigskip

\normalsize

Alternatively, without using the methods from this section, we can just make any plane 
quartic monic using~(\ref{mademonic}), then lift naively to characteristic~$0$ and try to use this
lift as input for \verb{pcc{. This way, we obtain the following three tables. 

\bigskip 

\scriptsize
\tabcolsep=0.235cm
\noindent \begin{tabular}{r||r|r|r}
            & time     & space & fails  \\
$p$         & pcc(s)   & (Mb)  & /1000  \\
\hline \hline
$11$        & $0.4$    & $  32$ & 225   \\       
$67$        & $1.3$    & $  32$ &  52   \\       
$521$       & $8.7$    & $  76$ &   5   \\       
$4099$      & $83 $    & $ 307$ &   1   \\      
$32771$     &$1153$    & $2086$ &   0         
\end{tabular}
\quad
\begin{tabular}{r||r|r|r}
             & time    & space  & fails \\
$q$          & pcc(s)  & (Mb)   & /1000 \\
\hline \hline
$3^5$        & $6.1$   & $ 32$  & $13$  \\       
$7^5$        & $14$    & $ 32$  & $0$   \\       
$17^5$       & $32$    & $ 80$  & $0$   \\       
$37^5$       & $71$    & $156$  & $0$   \\      
$79^5$       &$161$    & $288$  & $0$       
\end{tabular}
\quad
\begin{tabular}{r||r|r|r}
                & time    & space   & fails \\
$q$             & pcc(s)  & (Mb)    & /1000 \\
\hline \hline
$3^{10}$        &  $42$   &  $76$   & $0$   \\       
$7^{10}$        &  $94$   & $124$   & $0$   \\       
$17^{10}$       & $248$   & $320$   & $0$   \\       
$37^{10}$       & $524$   & $589$   & $0$   \\      
$79^{10}$       &$1296$   &$1311$   & $0$       
\end{tabular}
\bigskip 
\normalsize

Comparing the different tables, we see that the approach described in this section saves a factor of about $3$ in runtime and a factor of about $2$ in memory usage. 
Moreover, for small fields the naive lift of a plane quartic sometimes does not satisfy \cite[Ass.\,1]{tuitman2}, while this never seems to be the case for 
the lift constructed using our methods.

\subsection{Curves of genus four} \label{section_genus4}

\subsubsection{Lifting curves of genus four} \label{section_genus4lifting}

By \cite[Ex.\,IV.5.2.2]{hartshorne} the ideal of a canonical model 
$\overline{C} \subset \PPq^3 = \text{Proj} \, \FF_q[X,Y,Z,W]$
of a non-hyperelliptic genus $g = 4$ curve is generated by a cubic $\overline{S}_3$ and a unique quadric $\overline{S}_2$. Since $q$ is assumed odd, the latter
can be written as
 \[  \begin{pmatrix} X & Y & Z & W \end{pmatrix} \cdot \overline{M} \cdot \begin{pmatrix} X & Y & Z & W  \end{pmatrix}^t, \qquad \overline{M} \in \FF_q^{4 \times 4}, \ \overline{M}^t = \overline{M}. \]
Let $\chi_2 : \FF_q \rightarrow \{0, \pm 1\}$ denote the quadratic character on $\FF_q$. 
Then $\chi_2(\det \overline{M})$ is an invariant of $\overline{C}$, which 
is called the discriminant.

If we let $S_2, S_3 \in \mathcal{O}_K[X,Y,Z,W]$ be homogeneous polynomials that
reduce to $\overline{S}_2$ and $\overline{S}_3$ modulo $p$, then by \cite[Ex.\,IV.5.2.2]{hartshorne}
these define a genus $4$ curve $C \subset \PPK^3$ over $K$, thereby addressing (i) and (ii). However 
as mentioned in Section~\ref{section_firstfacts} we
expect the $K$-gonality of $C$ to be typically $2g - 2 = 6$.
This exceeds the $\FF_q$-gonality of $\overline{C}$:

\begin{lemma} \label{genus4gonality}
Let $\overline{C} / \FF_q$ be a non-hyperelliptic curve of genus $4$ and $\FF_q$-gonality
$\gamma$, and assume that $q$ is odd. 
If the discriminant of $\overline{C}$ is $0$ or $1$ then $\gamma = 3$. 
If it is $-1$ and $\# \overline{C}(\FF_{q^2}) > 0$ (which is guaranteed if $q > 7$) then $\gamma = 4$. 
Finally, if it is $-1$ and $\# \overline{C}(\FF_{q^2}) = 0$ then $\gamma = 5$.  
\end{lemma}
\begin{proof}
By \cite[Ex.\,IV.5.5.2]{hartshorne} our curve carries one or two geometric $g^1_3$'s, depending on whether
the quadric $\overline{S}_2$ is singular (discriminant $0$) or not. In the former case the quadric is a cone, and the $g^1_3$ corresponds
to projection from the top. This is automatically defined over $\FF_q$. In the latter case the quadric is $\FF_{q^2}$-isomorphic
to the hyperboloid $\PPq^1 \times \PPq^1 \subset \PPq^3$
and the $g^1_3$'s correspond to the two rulings of the latter. If the isomorphism can be defined over $\FF_q$ (discriminant $1$)
then the $g^1_3$'s
are $\FF_q$-rational. In the other case (discriminant $-1$) the smallest
field of definition is $\FF_{q^2}$.
So we can assume that the discriminant of $\overline{C}$ is $-1$, and therefore that $\gamma > 3$. Now suppose that $\# \overline{C}(\FF_{q^2}) > 0$, which is guaranteed if $q > 7$ by \cite[Thm.\,2]{howe}.
If there is an $\FF_q$-point then let $\overline{\ell}$ be the tangent line to $\overline{C}$ at it. In the other case 
we can find two conjugate $\FF_{q^2}$-points, and we let $\overline{\ell}$ be the line connecting both. In both cases $\overline{\ell}$ is defined
over $\FF_q$, and the pencil of planes
through $\overline{\ell}$ cuts out a $g^1_4$, as wanted. The argument can be reversed: if there exists a $g^1_4$ containing
an effective $\FF_q$-rational divisor $D$, then by Riemann-Roch 
we find that $|K - D|$ is non-empty. In particular there exists an effective $\FF_q$-rational divisor of degree $\deg (K-D) = 2$ on $\overline{C}$, and 
$\# \overline{C}(\FF_{q^2}) > 0$. So if $\# \overline{C}(\FF_{q^2}) = 0$ then $\gamma > 4$. Now 
note that $\# \overline{C}(\FF_{q^5}) > 0$ by the Weil bound. So $\overline{C}$ carries an effective divisor
$D$ of degree $5$. The linear system $|K-D|$ must be empty, for otherwise there would exist an $\FF_q$-point on $\overline{C}$.
But then Riemann-Roch implies that $\dim |D| = 1$, i.e.\ our curve carries an $\FF_q$-rational $g^1_5$.
\end{proof}

\begin{remark}
An example of a genus four curve $\overline{C}/\FF_3$ having
$\FF_3$-gonality five can be found in an \texttt{arXiv} version of our paper (\texttt{1605.02162v2}).
\end{remark}



To address Problem~\ref{liftingproblem} in the non-hyperelliptic genus $4$ case we make a case-by-case analysis.
 
 \paragraph*{\underline{$\chi_2(\det \overline{M}_2) = 0$}}
 
 In this case $\overline{S}_2$ is a cone over a conic. A linear change of variables takes 
 $\overline{S}_2$ to the form $ZW - X^2$, which we note is one of the standard realizations
 inside $\PPq^3$ of the weighted projective plane $\PPq(1,2,1)$. 
 It is classical how to find such a linear change of variables (diagonalization, essentially).
 Projecting from $(0:0:0:1)$ on the $XYZ$-plane amounts to eliminating the variable $W$, to obtain
 \begin{equation} \label{genus4conic}
 Z^3 \overline{S}_3(X, Y , Z, \frac{X^2}{Z}) = \overline{S}_3(XZ,YZ,Z^2,X^2).
 \end{equation}
 After dehomogenizing with respect to $Z$, renaming $X \leftarrow x$ and $Y \leftarrow y$ and rescaling if needed, we obtain an affine equation
$\overline{f} = y^3 + \overline{f}_2(x)y^2 + \overline{f}_4(x)y + \overline{f}_6(x)$,
 with $\overline{f}_i \in \FF_q[x]$ of degree at most $i$. Its Newton polygon
  is contained in (and typically equals):
\begin{center}
  \polfig{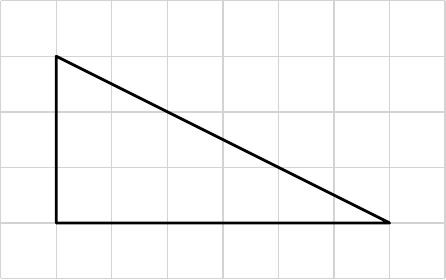}{3.2}{2}{$\Delta_{4,0}^0$}
\end{center}
  So Baker's bound is attained and we take for $f \in \mathcal{O}_K[x,y]$ a naive coefficient-wise lift.

 \paragraph*{\underline{$\chi_2(\det \overline{M}_2) = 1$}}
 
  In this case $\overline{S}_2$ is a hyperboloid. A linear change of variables takes $\overline{S}_2$ 
  to the standard form $XY - ZW$, which we note is the image of $\PPq^1 \times \PPq^1$ in $\PPq^3$ under the Segre embedding.
  Projection from $(0:0:0:1)$ on the $XYZ$-plane amounts to eliminating the variable $W$, to obtain
  \[  Z^3 \overline{S}_3(X,Y,Z,\frac{XY}{Z} ) = \overline{S}_3(XZ, YZ, Z^2, XY).\]
  After dehomogenizing with respect to $Z$ and renaming $X \leftarrow x$ and $Y \leftarrow y$ we obtain an affine equation
  $\overline{f} = \overline{f}_0(x) y^3 + \overline{f}_1(x) y^2 + \overline{f}_2(x) y + \overline{f}_3(x)$
  with all $\overline{f}_i \in \FF_q[x]$ of degree at most $3$. Its Newton polygon is contained in (and typically equals)
\begin{center}
  \polfig{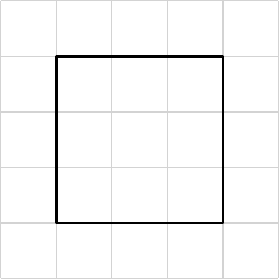}{2}{2}{$\Delta_{4,1}^{0}$}
\end{center}
So Baker's bound is attained and we can take for $f \in \mathcal{O}_K[x,y]$ a coefficient-wise lift of $\overline{f}$.

\paragraph*{\underline{$\chi_2(\det \overline{M}_2) = -1$}}

This is our first case where in general no plane model can be found for which Baker's bound is attained \cite[\S6]{CV}. 
If $\overline{C}(\FF_{q^2}) = \emptyset$, or in other words if $\gamma = 5$, then unfortunately we do not know how to
address Problem~\ref{liftingproblem}. 
We therefore assume that $\overline{C}(\FF_{q^2}) \neq \emptyset$ and hence that $\gamma = 4$.
This is guaranteed if $q > 7$, so for point counting purposes this is amply sufficient.
We follow the proof of Lemma~\ref{genus4gonality}: by exhaustive search we find a point $P \in \overline{C}(\FF_{q^2})$
along with its Galois conjugate $P'$ and consider the line $\overline{\ell}$ connecting both (tangent line if $P = P'$). This line is defined over $\FF_q$, so that modulo a projective transformation we can assume that $\overline{\ell} : X = Z = 0$.

When plugging in $X = Z = 0$ in $\overline{S}_2$ we find a non-zero quadratic expression in $Y$ and $W$. Indeed:
$\overline{S}_2$ cannot vanish identically on $\overline{\ell}$ because no three points of $\overline{S}_2(\FF_q)$ are collinear. 
Because $\overline{C}$ intersects $\overline{\ell}$ in two points (counting multiplicities) we find that
\[ \overline{S}_3(0,Y,0,W) = (\overline{a}Y + \overline{b}W) \overline{S}_2(0,Y,0,W) \]
for certain $\overline{a}, \overline{b} \in \FF_q$ that are possibly zero. Lift $\overline{S}_2$ coefficient-wise to
a homogenous quadric $S_2 \in \mathcal{O}_K[X,Y,Z,W]$ and let $a,b \in \mathcal{O}_K$ reduce to $\overline{a},\overline{b}$ mod $p$.
We now construct $S_3 \in \mathcal{O}_K[X,Y,Z,W]$ as follows: for the coefficients at $Y^3, Y^2W, YW^2, W^3$ we make the unique choice for which
\[ S_3(0,Y,0,W) = (aY + bW) S_2(0,Y,0,W), \]
while the other coefficients are randomly chosen lifts of the corresponding coefficients of $\overline{S}_3$.
Then the genus $4$ curve $C \subset \PPK^3$ defined by $S_2$ and $S_3$ is of gonality $4$. Indeed, it is constructed such that the line $\ell : X = Z = 0$ intersects the curve in two points (possibly over a quadratic extension), and the pencil of planes through this line
cuts out a $g^1_4$.

Now we project our lift $C \subset \PPK^3$ from $(0:0:0:1)$ to a curve in $\PPK^2$. This amounts to eliminating $W$ from $S_2$ and $S_3$.
By dehomogenizing the resulting sextic with respect to $Z$, and by 
renaming $X \leftarrow x$ and $Y \leftarrow y$ we end up with a polynomial $f \in \mathcal{O}_K[x,y]$ whose Newton polygon
is contained in (and typically equals):
\begin{center}
  \polfig{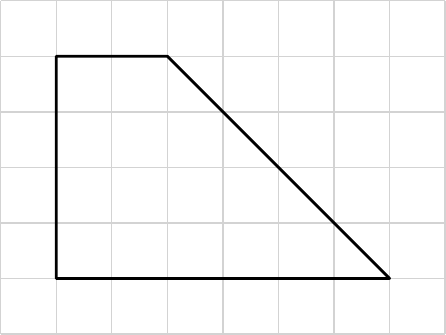}{3.2}{2.4}{$\Delta_{4,-1}^{6}$}
\end{center}
Geometrically, what happens is that the points of $C$ on $\ell$ are both mapped to $(0:1:0)$ under projection from $(0:0:0:1)$, creating a singularity there, which in terms of the Newton polygon results in $6\Sigma$ with its top chopped off.
The polynomial $f$ satisfies (i), (ii) and (iii) from Problem~\ref{liftingproblem}. Note that Baker's bound is usually \emph{not} 
attained here: it gives an upper bound of $9$, while $C$ has genus $4$. 
So it is crucial to lift the equations to $\mathcal{O}_K$ \emph{before} projecting on the plane.

\noindent \varhrulefill[0.4mm]
\vspace{-0.3cm}

\begin{algo} \label{algorithm_genus4}
Lifting curves of genus $4$: basic solution 

\vspace{-0.2cm}
\noindent \varhrulefill[0.4mm]

\noindent \textbf{Input:} non-hyperelliptic genus $4$ curve $\overline{C}/\FF_q$ of $\FF_q$-gonality $\gamma \leq 4$

\noindent \textbf{Output:} lift $f \in \mathcal{O}_K[x,y]$ satisfying (i), (ii), (iii) that is supported

\noindent \qquad \qquad \qquad $\bullet$ on $\Delta_{4,0}^0$ if the discriminant is $0$, or else

\noindent \qquad \qquad \qquad $\bullet$ on $\Delta_{4,1}^0$ if the discriminant is $1$, or else

\noindent \qquad \qquad \qquad $\bullet$ on $\Delta_{4,-1}^6$ 

\vspace{-0.2cm}
\noindent \varhrulefill[0.4mm]

\noindent \small \phantom{0}1 \normalsize: $\overline{C} \gets \text{CanonicalImage}(\overline{C})$ in $\PPq^3 = \proj \FF_q[X,Y,Z,W]$

\noindent \small \phantom{0}2 \normalsize: $\overline{S}_2 \gets \text{unique quadric in Ideal}(\overline{C})$; $\overline{M}_2 \gets \text{Matrix}(\overline{S}_2)$; $\chi \gets \chi_2(\det \overline{M}_2)$

\noindent \small \phantom{0}3 \normalsize: $\overline{S}_3 \gets \text{cubic that along with } \overline{S}_2 \text{ generates Ideal}(\overline{C})$

\noindent \small \phantom{0}4 \normalsize: \textbf{if} $\chi = 0$ \textbf{then} 

\noindent \small \phantom{0}5 \normalsize: \quad apply automorphism of $\PPq^3$ transforming $\overline{S}_2=0$ into $ZW - X^2=0$

\noindent \small \phantom{0}6 \normalsize: \quad \textbf{return} NaiveLift(Dehomogenization${}_Z$($\overline{S}_3(XZ,YZ,Z^2,X^2)$))

\noindent \small \phantom{0}7 \normalsize: \textbf{else if} $\chi = 1$ \textbf{then}

\noindent \small \phantom{0}8 \normalsize: \quad apply automorphism of $\PPq^3$ transforming $\overline{S}_2 = 0$ into $XY-ZW = 0$

\noindent \small \phantom{0}9 \normalsize: \quad \textbf{return} NaiveLift(Dehomogenization${}_Z$($\overline{S}_3(XZ,YZ,Z^2,XY)$))

\noindent \small 10 \normalsize: \textbf{else}

\noindent \small 11 \normalsize: \quad $P := \text{Random}(\overline{C}(\FF_{q^2}))$; $P' := \text{Conjugate}(P)$

\noindent \small 12 \normalsize: \quad $\overline{\ell} \gets \text{line through $P$ and $P'$ (tangent line if $P = P'$)}$

\noindent \small 13 \normalsize: \quad  apply automorphism of $\PPq^3$ transforming $\overline{\ell}$ into $X = Z = 0$

\noindent \small 14 \normalsize: \quad  $S_2 \leftarrow \text{NaiveLift}(\overline{S}_2)$

\noindent \small 15 \normalsize: \quad  $S_3 \leftarrow$ lift of $\overline{S}_3$ satisfying $S_3(0,Y,0,W) = (aY + bW)S_2(0,Y,0,W)$ for $a,b \in \mathcal{O}_K$

\noindent \small 16 \normalsize: \quad  \textbf{return} Dehomogenization${}_Z$(res${}_W$($S_2,S_3$))

\vspace{-0.2cm}
\noindent \varhrulefill[0.4mm]
\end{algo}

\subsubsection{Optimizations} \label{optim_genus4}

 \paragraph*{\underline{$\chi_2(\det \overline{M}_2) = 0$}}

By applying (\ref{mademonic}) to 
a polynomial with Newton polygon $\Delta_{4,0}^{0}$
one ends up with a polynomial that is monic in $y$ and that has degree $6$ in $x$. This can be improved as soon as $\overline{C}(\FF_q)\neq 0$, which is guaranteed if $q > 49$ by \cite[Thm.\,2]{howe}. Namely
we can view (\ref{genus4conic}) as the defining equation of a smooth curve in the weighted
  projective plane $\PPq(1, 2, 1)$.
  Using an automorphism of the latter we can position
  a given $\FF_q$-rational point $P$ at $(1:0:0)$ and the corresponding tangent line at $X = 0$, in order to end up with a Newton polygon that is contained in (and typically equals): 
  \begin{center}
  \polfig{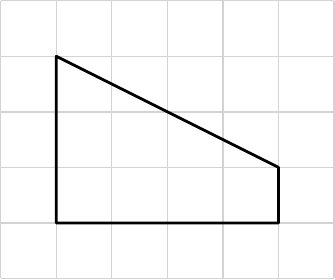}{2.4}{2}{$\Delta_{4,0}^{1}$}
\end{center}
  See Remark~\ref{autsofP121} below for how to do this in practice.
  So we find $\deg_x f = 4$, which is optimal because the $g^1_3$ is unique in the case of a singular $\overline{S}_2$. 
  There is a caveat here, in that the tangent line at $P$ might exceptionally be vertical, i.e.\ $P$ might be a ramification point
  of our degree $3$ map $(x,y) \mapsto x$. In this case it is impossible to position this line at $X = 0$, but 
  in practice one can simply retry with another $P$.
  But in fact having a vertical tangent line is an even slightly better situation, as explained in Remark~\ref{genus4chi0remark} below.

   \begin{remark}~\label{autsofP121}
   The automorphisms of $\PPq(1,2,1)$ can be applied directly to $\overline{f}$. They correspond to
   \begin{itemize}
     \item substituting 
   $y \leftarrow \overline{a} y + \overline{b} x^2 + \overline{c} x + \overline{d}$ and $x \leftarrow \overline{a}' x + \overline{b}'$ in $
  \overline{f}$ for some 
  $\overline{a},\overline{a}' \in \FF_q^\ast$ and $\overline{b},\overline{b}',\overline{c},\overline{d} \in \FF_q$,
    \item exchanging the line at infinity for the $y$-axis by replacing $\overline{f}$ by $x^6 \overline{f}(x^{-1},x^{-2}y)$, 
  \end{itemize} or to
  a composition of both. For instance imagine that an affine point $P = (\overline{a},\overline{b})$ was found with a non-vertical
  tangent line. Then $\overline{f} \leftarrow \overline{f}(x + \overline{a}, y + \overline{b})$ translates this point to the origin,
  at which the tangent line becomes of the form $y = \overline{c} x$. Substituting $\overline{f} \leftarrow \overline{f}(x,y + \overline{c}x)$ positions this line horizontally, and finally replacing $\overline{f}$ by $x^6 \overline{f}(x^{-1},x^{-2}y)$ results
  in a polynomial with Newton polygon contained in $\Delta_{4,0}^1$.
  \end{remark}

  \begin{remark}[non-generic optimizations] \label{genus4chi0remark}
  If $P$ has a vertical tangent line then positioning it at $(1:0:0)$
  results in a Newton polygon that is contained in (and typically equals) the first polygon below:
    \begin{center}
  \polfig{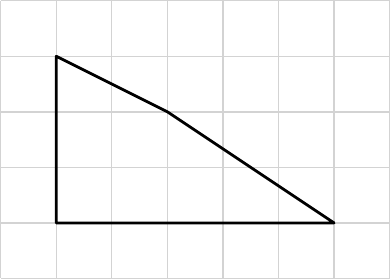}{2.8}{2}{$\Delta_{4,0}^{2}$}
  \polfig{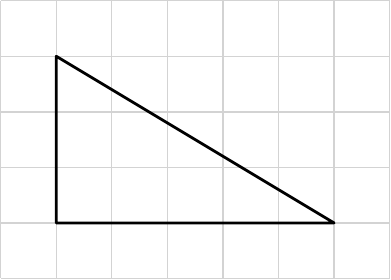}{2.8}{2}{$\Delta_{4,0}^{3}$}
\end{center}
  Even though $\deg_x f = 5$ here, this results in a slightly faster point count.
  Such a $P$ will exist if and only if the ramification scheme of $(x,y) \mapsto x$ has an $\FF_q$-rational point. 
  Following the same heuristic as in Remark~\ref{genus3flexremark} we expect that this works in about $1 - 1/e$ of the cases.
  If there exists a point of ramification index $3$ then one can even end up inside the second polygon. This event is highly exceptional,
  but we include it in our discussion because this corresponds to the well-known class
  of $C_{3,5}$ curves.
  \end{remark}

 \paragraph*{\underline{$\chi_2(\det \overline{M}_2) = 1$}}

By applying (\ref{mademonic}) to a polynomial with Newton polygon $\Delta_{4,1}^{0}$ one ends up with a polynomial
that is monic in $y$ and that has degree $3 + (\gamma - 1)3 = 9$ in $x$. This
can be improved as soon as $\overline{C}(\FF_q) \neq 0$, which is guaranteed if $q > 49$ by \cite[Thm.\,2]{howe}.
Assume as before that $\overline{S}_2$ is in the standard form $XY - ZW$. So it is the image of the Segre embedding
\begin{equation}  \label{segre}
  \PPq^1 \times \PPq^1 \hookrightarrow \PPq^3 : ( (X_0 : Z_0), (Y_0 : W_0) ) \mapsto (X_0W_0 : Y_0Z_0 : Z_0W_0 : X_0Y_0 ).
\end{equation}
That is: we can view $\overline{C}$ as the curve in $\mathbb{P}^1 \times \mathbb{P}^1$
defined by the bihomogeneous polynomial
\[ \overline{S}_3(X_0W_0,Y_0Z_0,Z_0W_0,X_0Y_0) \]
of bidegree $(3,3)$. Remark that if we dehomogenize with respect to both $Z_0$ and $W_0$ and rename
$X_0 \leftarrow x$ and $Y_0 \leftarrow y$ then we get the polynomial $\overline{f}$ from before.
Now if our curve has a rational point $P$, by applying an appropriate projective transformation 
in each component we can arrange that $P = ((1:0), (1:0))$. If we then dehomogenize we end up with a Newton polygon that is contained in (and typically equals):
\begin{center}
  \polfig{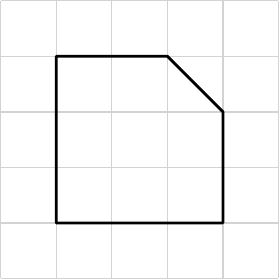}{2}{2}{$\Delta_{4,1}^{1}$}
\end{center}
So Baker's bound is attained and we take for $f \in \mathcal{O}_K[x,y]$ a naive coefficient-wise lift. Now applying (\ref{mademonic})
typically results in a polynomial of degree $3 + (\gamma - 1)2 = 7$ in $x$.

  \begin{remark}~\label{autsofP1P1}
   The automorphisms of $\PPq^1 \times \PPq^1$ can again be applied directly to $\overline{f}$. They correspond to
   \begin{itemize}
     \item substituting 
   $y \leftarrow \overline{a} y + \overline{b}$ and $x \leftarrow \overline{a}' x + \overline{b}'$ in $
  \overline{f}$ for some 
  $\overline{a},\overline{a}' \in \FF_q^\ast$ and $\overline{b},\overline{b}' \in \FF_q$,
    \item exchanging the $x$-axis for the horizontal line at infinity by replacing $\overline{f}$ by $y^3 \overline{f}(x,y^{-1})$,
    \item exchanging the $y$-axis for the vertical line at infinity by replacing $\overline{f}$ by $x^3 \overline{f}(x^{-1},y)$, 
  \end{itemize} or to
  a composition of these. For instance imagine that an affine point $P = (\overline{a}, \overline{b})$ was found, then
  $\overline{f} \leftarrow \overline{f}(x + \overline{a}, y + \overline{b})$ translates this point to the origin, and
  subsequently replacing $\overline{f}$ by $x^3y^3\overline{f}(x^{-1}, y^{-1})$ results in a polynomial with Newton polygon contained
  in $\Delta_{4,1}^1$.
  \end{remark}

\begin{remark}[non-generic optimizations]
If one manages to let $P$ be a point with a horizontal tangent line, i.e.\ if $P$ is a ramification point
of the projection map from $\overline{C}$ onto the second component of $\PPq^1 \times \PPq^1$, then the Newton polygon even becomes contained
in (and typically equals):
\begin{center}
  \polfig{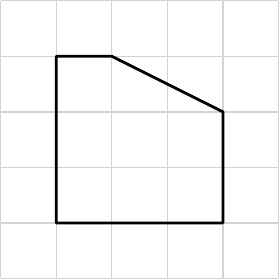}{2}{2}{$\Delta_{4,1}^{2}$}
\end{center}
This eventually results in a polynomial $f \in \mathcal{O}_K[x,y]$ of degree $3 + (\gamma - 1)1 = 5$ in $x$. As in the discriminant $0$ case,
we heuristically expect the probability of success to be about $1 - 1/e$.
However, it is also fine to find a ramification point of the projection of $\overline{C}$ onto the first component of $\PPq^1 \times \PPq^1$,
because we can change the role of
$(X_0,Z_0)$ and $(Y_0,W_0)$ if wanted. 
Assuming independence of events, the percentage of non-hyperelliptic genus $4$ curves with discriminant $1$ that admit a Newton polygon of the form $\Delta_{4,1}^2$ should be approximately $1 - 1/e^2$.
\end{remark}

 \paragraph*{\underline{$\chi_2(\det \overline{M}_2) = -1$}}

By applying (\ref{mademonic}) to a polynomial with Newton polygon $\Delta_{4,-1}^{6}$ we end up with a polynomial 
that is monic in $y$ and that has degree $3 + (\gamma - 1)2 = 9$. This
can be improved as soon as $\overline{C}(\FF_q) \neq 0$, which is guaranteed if $q > 49$ by \cite[Thm.\,2]{howe}.
In this case we redo the construction with $\overline{\ell}$ the tangent line to a point $P \in \overline{C}(\FF_q)$. As before we apply a projective transformation to
obtain $\overline{\ell} : X = Z = 0$, but in addition we make sure that $P = (0:0:0:1)$.
This implies that $\overline{S}_2(0,Y,0,W) = Y^2$, possibly after multiplication by a scalar. We now proceed as before, to find
lifts $S_2, S_3 \in \mathcal{O}_K[X,Y,Z,W]$ that cut out a genus $4$ curve $C \subset \PPK^3$, still
satisfying the property of containing $(0:0:0:1)$ with corresponding tangent line $\ell : X = Z = 0$. If we then
project from $(0:0:0:1)$ we end up with a quintic in $\PPK^2$, rather than a sextic. The quintic still passes through the point $(0:1:0)$, which is now non-singular: otherwise the pencil of lines through that point would cut out a $K$-rational $g^1_3$. We can therefore apply a projective transformation over $K$ that maps the corresponding tangent line to infinity, while keeping the point at $(0:1:0)$. After having done so, we dehomogenize to find a polynomial $f \in \mathcal{O}_K[x,y]$ whose Newton polygon is contained in (and typically equals)
\begin{center}
  \polfig{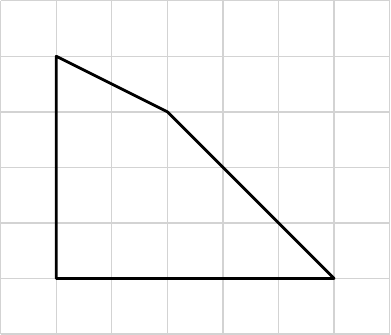}{2.8}{2.4}{$\Delta_{4,-1}^{5}$}
\end{center}
It still satisfies (i), (ii) and (iii), while here $\deg_xf \leq 5$.

\subsubsection{Implementation}

The tables below contain timings, memory usage and failure rates for $\chi_2=0,1,-1$ and various values of $p$ and $q=p^n$. 
For the precise meaning of the various entries in the table see Section~\ref{section_genus3implementationandtimings}.\\


\vspace{-0.2cm}

\noindent \textbf{$\mathbf{\chi_2=0}$}\\

\scriptsize
\tabcolsep=0.11cm
\noindent \begin{tabular}{r||r|r|r|r}
             & time     &time    & space  & fails   \\
$p$          & lift(s)  &pcc(s)  & (Mb)   & /1000   \\
\hline \hline
$11$         & $0.01$   & $0.3$  & $32$   & $159$   \\
$67$         & $0.01$   & $1.4$  & $32$   & $2$     \\
$521$        & $0.01$   & $13$   & $73$   & $2$     \\
$4099$       & $0.01$   & $189$  & $323$  & $0$     \\
$32771$      & $0.01$   &$2848$  & $2396$ & $0$    
\end{tabular}
\quad
\begin{tabular}{r||r|r|r|r}
             & time     &time    & space  & fails  \\
$q$          & lift(s)  &pcc(s)  & (Mb)   & /1000  \\
\hline \hline
$3^5$        &  $0.04$  &$6.6$   & $64$   & $2$    \\
$7^5$        &  $0.05$  & $13$   & $73$   & $0$    \\
$17^5$       &  $0.1$   & $32$   &$118$   & $0$    \\
$37^5$       &  $0.1$   & $73$   &$197$   & $0$    \\
$79^5$       &  $0.1$   &$183$   &$371$   & $0$   
\end{tabular}
\quad
\begin{tabular}{r||r|r|r|r}
             & time   &time     & space  & fails  \\
$q$          & lift(s)&pcc(s)   & (Mb)   & /1000  \\
\hline \hline
$3^{10}$     & $0.3$  & $34$    &$112$   & $0$   \\
$7^{10}$     & $0.4$  & $76$    &$156$   & $0$   \\
$17^{10}$    & $0.6$  &$205$    &$320$   & $0$   \\
$37^{10}$    & $0.7$  &$537$    &$653$   & $0$   \\
$79^{10}$    & $0.9$  &$1392$   &$1410$  & $0$  
\end{tabular}\\

\normalsize


\vspace{0.3cm}

\noindent \textbf{$\mathbf{\chi_2=1}$}\\

\noindent \scriptsize
\tabcolsep=0.11cm
\begin{tabular}{r||r|r|r|r}
             & time     &time    & space  & fails   \\
$p$          & lift(s)  &pcc(s)  & (Mb)   & /1000   \\
\hline \hline
$11$         & $0.01$   & $0.4$  & $32$   & $169$   \\
$67$         & $0.02$   & $1.8$  & $32$   & $1$     \\
$521$        & $0.02$   & $14$   & $76$   & $0$     \\
$4099$       & $0.02$   & $230$  & $508$  & $0$     \\
$32771$      & $0.02$   &$2614$  &$3616$  & $0$    
\end{tabular}
\quad
\begin{tabular}{r||r|r|r|r}
             & time     &time    & space  & fails  \\
$q$          & lift(s)  &pcc(s)  & (Mb)   & /1000  \\
\hline \hline
$3^5$        & $0.1$    & $7.5$  & $64$   & $0$    \\
$7^5$        & $0.1$    & $16$   & $112$  & $0$    \\
$17^5$       & $0.2$    & $41$   & $197$  & $0$    \\
$37^5$       & $0.2$    & $94$   & $320$  & $0$    \\
$79^5$       & $0.2$    &$241$   & $589$  & $0$   
\end{tabular}
\quad
\begin{tabular}{r||r|r|r|r}
             & time   &time     & space  & fails  \\
$q$          & lift(s)&pcc(s)   & (Mb)   & /1000  \\
\hline \hline
$3^{10}$     & $0.7$  & $41$    & $150$   & $0$   \\
$7^{10}$     & $1.2$  & $102$   & $320$   & $0$   \\
$17^{10}$    & $2.1$  & $276$   & $556$   & $0$   \\
$37^{10}$    & $2.8$  & $736$   & $1070$  & $0$   \\
$79^{10}$    & $3.9$  & $1904$  & $2016$  & $0$  
\end{tabular}\\

\normalsize


\vspace{0.3cm}

\noindent \textbf{$\mathbf{\chi_2=-1}$}\\

\noindent \scriptsize
\tabcolsep=0.11cm
\begin{tabular}{r||r|r|r|r}
             & time     &time    & space  & fails   \\
$p$          & lift(s)  &pcc(s)  & (Mb)   & /1000   \\
\hline \hline
$11$         & $0.06$   & $2.4$  &  $73$  & $0$     \\
$67$         & $0.02$   & $4.3$  &  $73$  & $0$     \\
$521$        & $0.02$   & $32$   & $124$  & $0$     \\
$4099$       & $0.03$   &$503$   & $815$  & $0$     \\
$32771$      & $0.02$   &$5958$  & $6064$ & $0$    
\end{tabular}
\quad
\begin{tabular}{r||r|r|r|r}
             & time     &time    & space  & fails  \\
$q$          & lift(s)  &pcc(s)  & (Mb)   & /1000  \\
\hline \hline
$3^5$        & $0.15$   & $20$   & $76$   & $0$    \\
$7^5$        & $0.3$    & $46$   &$156$   & $0$    \\
$17^5$       & $0.4$    &$108$   &$241$   & $0$    \\
$37^5$       & $0.6$    &$243$   &$403$   & $0$    \\
$79^5$       & $0.8$    &$570$   &$749$   & $0$   
\end{tabular}
\quad
\begin{tabular}{r||r|r|r|r}
             & time    &time     & space  & fails  \\
$q$          & lift(s) &pcc(s)   & (Mb)   & /1000  \\
\hline \hline
$3^{10}$     & $1.3$   & $130$   & $273$  & $0$   \\
$7^{10}$     & $2.8$   & $312$   & $416$  & $0$   \\
$17^{10}$    & $5.0$   & $815$   & $813$  & $0$   \\
$37^{10}$    & $6.5$   &$1939$   & $1463$ & $0$   \\ 
$79^{10}$    & $8.4$   &$4609$   & $2942$ & $0$   \\
\end{tabular}
\normalsize

\bigskip

\par Contrary to the genus~$3$ case, we see that for very small $p$ or $q=p^n$, sometimes we do not find a lift satisfying
\cite[Ass.\,1]{tuitman2}. However, in these cases we can usually compute the zeta function by counting points naively, 
so not much is lost here in practice. Note that the point counting is considerably slower for $\chi_2=-1$ than for 
$\chi_2=0,1$ which is due to the map from the curve to $\PPq^1$ having degree $4$ instead of $3$ in this case.

\subsection{Curves of genus five} \label{section_genus5}

\subsubsection{Lifting curves of genus five} \label{section_genus5lifting}

By Petri's theorem \cite{saintdonat} a minimal set of generators for the ideal of a canonical model 
\[ \overline{C} \subset \PPq^4 = \text{Proj} \, \FF_q[X,Y,Z,W,V] \]
of a non-hyperelliptic genus $5$ curve consists of
\begin{itemize}
  \item three quadrics $\overline{S}_{2,1}, \overline{S}_{2,2}, \overline{S}_{2,3}$ and two cubics $\overline{S}_{3,1}, \overline{S}_{3,2}$ in the trigonal case,
  \item just three quadrics $\overline{S}_{2,1}, \overline{S}_{2,2}, \overline{S}_{2,3}$ in the non-trigonal case. 
\end{itemize}
So given such a minimal set of generators, it is straightforward to decide trigonality. We denote the space of quadrics in the ideal
of $\overline{C}$ by $\mathcal{I}_2(\overline{C})$. Then in both settings $\mathcal{I}_2(\overline{C})$ is a three-dimensional $\FF_q$-vector space
of which $\overline{S}_{2,1}, \overline{S}_{2,2}, \overline{S}_{2,3}$ form a basis.

\paragraph*{Trigonal case} \hfill Here Petri's theorem moreover tells us that $\mathcal{I}_2(\overline{C})$ cuts out a
smooth ir-

\noindent \begin{minipage}[b]{9.5cm}
reducible surface $\overline{S}$ that is a
rational normal surface scroll of type $(1,2)$. This means that up to a linear change of variables,  
it is the image $\overline{S}(1,2)$ of
\[ \PPq^1 \times \PPq^1 \hookrightarrow \PPq^4 : ((s:t),(u:v)) \mapsto (vst:ut:vt^2:us:vs^2),\]
i.e.\ \hfill it is the ruled surface obtained by simultaneously pa-
\end{minipage}
\ \
\begin{minipage}[b]{4.1cm}
\begin{center}
  \includegraphics[width=5cm]{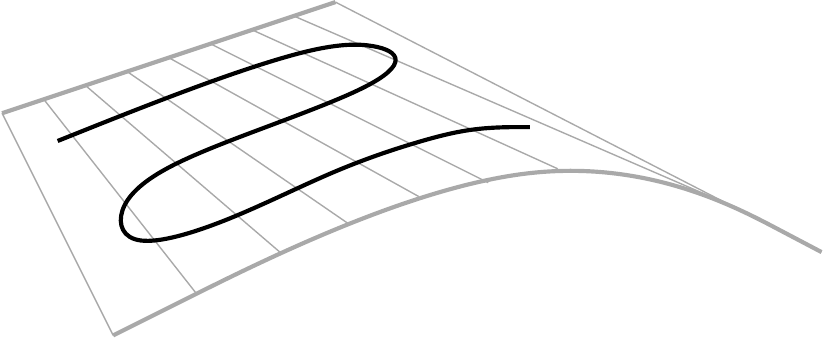}\\
  \vspace{0.4cm}
\end{center}
\end{minipage}

\noindent rameterizing a line in the $YW$-plane (called the directrix)
and a conic in the $XZV$-plane, each time drawing the rule through the points under consideration (each of these rules
intersects our trigonal curve in three points, counting multiplicities).
In other words, modulo a linear change of variables
the space $\mathcal{I}_2(\overline{C})$ admits the basis
\begin{equation} \label{defpols_scroll}
 X^2 - ZV, \qquad XY - ZW, \qquad XW - YV.
\end{equation}
Note that these are (up to sign) the $2 \times 2$ minors of
\[
  \left( \begin{array}{cc} X & V \\ Z & X \\ \end{array} \right| \hspace{-0.1cm}
  \left. \begin{array}{c} W \\ Y \\ \end{array} \right).
\]  
It is not trivial to \emph{find} such a linear change of variables. 
A general method using Lie algebras
for rewriting Severi-Brauer surfaces in standard form
was developed by de Graaf, Harrison, P\'ilnikov\'a and Schicho~\cite{GHPS}, and a Magma function
\texttt{ParametrizeScroll} for carrying out this procedure in the case of rational normal surface scrolls 
was written by Schicho. Unfortunately this was intended to work over fields of characteristic zero only, and 
indeed the function always seems to crash when invoked over fields of characteristic three; see also Remark~\ref{blinduse} below. 
We do
not know how fundamental this flaw is, or to what extent it is  an artefact
of the implementation, 
but to resolve this issue we have implemented an ad hoc method that is specific to scrolls of type $(1,2)$. It can
be found in \texttt{convertscroll.m}; more background on the underlying reasoning can be read in an \texttt{arXiv} version of this paper (1605.02162v2).

Once our quadrics $\overline{S}_{2,1}, \overline{S}_{2,2}, \overline{S}_{2,3}$ are given by \eqref{defpols_scroll} we project from the line $X = Y = Z = 0$, which amounts to eliminating the variables
$V$ and $W$, in order to obtain the polynomials
\[ \overline{S}_{3,i}^\text{pr} = Z^3 \overline{S}_{3,i}(X,Y,Z,\frac{X^2}{Z},\frac{XY}{Z}) = \overline{S}_{3,i}(XZ,YZ,Z^2,X^2,XY)  \]
for $i=1,2$.
Dehomogenizing with respect to $Z$ and renaming $X \leftarrow x$ and $Y \leftarrow y$ we obtain two polynomials $\overline{f}_1, \overline{f}_2 \in \FF_q[x,y]$, whose zero loci intersect in the curve defined by $\overline{f} = \gcd(\overline{f}_1,\overline{f}_2)$. The Newton polygon of $\overline{f}$ is contained in (and typically equals):
\begin{center}
\polfig{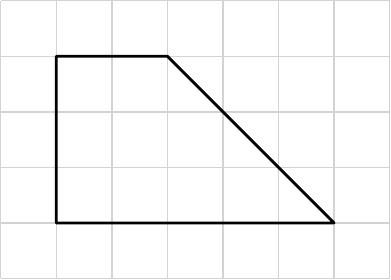}{2.8}{2}{$\Delta_{5,\text{trig}}^{0,0}$}
\end{center}
Note that in particular $\overline{f}$ attains Baker's bound, and a naive Newton polygon preserving lift $f \in \mathcal{O}_K[x,y]$ satisfies (i), (ii) and (iii). 
An alternative (namely, toric) viewpoint on our construction of $\overline{f}$, along with more background on the claims above, is given in Section~\ref{section_trigonal}.

\paragraph*{Non-trigonal case}

In the non-trigonal case, let us write the quadrics as
 \[ \overline{S}_{2,i} = \begin{pmatrix} X & Y & Z & W & V \end{pmatrix} \cdot \overline{M}_i \cdot \begin{pmatrix} X & Y & Z & W & V \end{pmatrix}^t, \qquad \overline{M}_i \in \FF_q^{5 \times 5}, \ \overline{M}_i^t = \overline{M}_i. \]
The curve $\mathfrak{D}(\overline{C})$ in $\PPq^2 = \text{Proj} \, \FF_q[\lambda_1, \lambda_2, \lambda_3]$ defined by
\[ \det (\lambda_1  \overline{M}_1 +  \lambda_2 \overline{M}_2 + \lambda_3 \overline{M}_3) = 0 \]
parameterizes the singular members of $\mathcal{I}_2(\overline{C})$. It is a possibly reducible curve called the discriminant curve of $\overline{C}$, known
to be of degree $5$ and having at most nodes as singularities \cite{cornalba}${}^\dagger$. The non-singular points correspond to quadrics of rank $4$, while
the nodes correspond to quadrics of rank $3$. For a point $P \in \mathfrak{D}(\overline{C})(\FF_q)$, let us denote by $\overline{M}_P$ the corresponding $(5 \times 5)$-matrix and by $\overline{S}_P$ the corresponding quadric, both of which are 
well-defined up to a scalar. We define
\[ \chi : \mathfrak{D}(\overline{C})(\FF_q) \rightarrow \{ 0, \pm 1 \} : P \mapsto \left\{ \begin{array}{ll} \chi_2(\pdet( \overline{M}_P) ) & \text{if $P$ is non-singular,} \\ 0 & \text{if $P$ is singular,} \\ \end{array} \right. \]
where $\pdet$ denotes the pseudo-determinant, i.e.\ the product of the non-zero eigenvalues.

If we let $S_{2,i} \in \mathcal{O}_K[X,Y,Z,W,V]$ be homogeneous polynomials that
reduce to $\overline{S}_{2,i}$ modulo $p$, then by \cite[Ex.\,IV.5.5.3]{hartshorne}
these define a genus $5$ curve $C \subset \PPK^4$ over $K$, thereby addressing (i) and (ii). But 
as mentioned in Section~\ref{section_firstfacts} we
expect the $K$-gonality of $C$ to be typically $2g - 2 = 8$,
which exceeds the $\FF_q$-gonality of $\overline{C}$:

\begin{lemma} \label{genus5gonality}
Let $\overline{C} / \FF_q$ be a non-hyperelliptic non-trigonal curve of genus $5$ and $\FF_q$-gonality $\gamma$, and assume that $q$ is odd.
If there is a point $P \in \mathfrak{D}(\overline{C})(\FF_q)$ for which $\chi(P) \in \{0, 1 \}$ then $\gamma = 4$.
If there does not exist such a point and $\# \overline{C}(\FF_{q^3}) > 0$ (which is guaranteed if $q > 3$) then $\gamma = 5$.
If there does not exist such a point and $\# \overline{C}(\FF_{q^3}) = 0$ then $\gamma = 6$.
\end{lemma}

\begin{proof}
By \cite[VI.Ex.\,F]{cornalba}${}^\dagger$ the geometric $g^1_4$'s are in correspondence with the singular quadrics containing $\overline{C}$.
More precisely:
\begin{itemize}
  \item Each rank $4$ quadric is a cone over $\PPq^1 \times \PPq^1$. By taking its span with the top, each line on $\PPq^1 \times \PPq^1$ gives rise to a plane intersecting the curve in $4$ points. By varying the line we obtain two $g^1_4$'s, one for each ruling of $\PPq^1 \times \PPq^1$.
  \item Each rank $3$ quadric is a cone with a $1$-dimensional top over a conic. By taking its span with the top, every point of the conic gives rise to a plane intersecting the curve in $4$ points. By varying the point we obtain a $g^1_4$.
\end{itemize}  
There are no other geometric $g^1_4$'s. Over $\FF_q$, we see that there exists a rational $g^1_4$ precisely 
\begin{itemize}
\item when there is a rank $4$ quadric that is defined over $\FF_q$, such that the base of the corresponding cone is $\FF_q$-isomorphic to $\PPq^1 \times \PPq^1$, or
\item when there is a rank $3$ quadric that is defined over $\FF_q$.
\end{itemize} 
In terms of the discriminant, this amounts to the existence of a $P \in \mathfrak{D}(\overline{C})$ for which $\chi(P) \in \{0, 1 \}$. So let us assume that $\gamma > 4$. If $\# \overline{C}(\FF_{q^3}) > 0$, which by the Serre-Weil bound is guaranteed
for $q > 3$, then there exists an effective $\FF_q$-rational degree $3$ divisor $D$ on $\overline{C}$. Because our curve is non-trigonal we
find $\dim |D| = 0$, so by the Riemann-Roch theorem we have
that $\dim | K - D| = 1$, and because $\deg (K-D) = 5$ we conclude that there exists a rational $g^1_5$ on $\overline{C}$. 
(Remark: geometrically, this $g^1_5$ is cut out by the pencil of hyperplanes through the plane spanned by the support of $D$, taking into account multiplicities.) The argument can be reversed: if there exists a $g^1_5 \ni D$ for some
$\FF_q$-rational divisor $D$ on $\overline{C}$, then Riemann-Roch implies that $|K-D|$ is non-empty, yielding an effective divisor of degree $3$, and in particular $\# \overline{C}(\FF_{q^3}) > 0$. So it remains to prove that if $\# \overline{C}(\FF_{q^3}) = 0$ then there exists a rational $g^1_6$. We make a case distinction: 
\begin{itemize}
  \item If $\# \overline{C}(\FF_{q^2}) > 0$ then there exists a rational effective divisor $D$ of degree $2$, and Riemann-Roch implies that $ \dim |K-D| = 2$, yielding the requested rational $g^1_6$ (even a $g^2_6$, in fact).
  \item If $\# \overline{C}(\FF_{q^2}) = 0$ then at least $\# \overline{C}(\FF_{q^6}) > 0$ by the Weil bound, so there exists a rational effective divisor 
  $D$ of degree $6$. Then $K-D$ is of degree $2$ and by our assumption $|K - D|$ is empty. But then Riemann-Roch asserts that $\dim |D| = 1$, and we have our rational $g^1_6$.
\end{itemize}
This ends the proof.
\end{proof}


\begin{remark} \label{remarkjeroen}
If $q$ is large enough then it is very likely that $\mathfrak{D}(\overline{C})(\FF_q)$ will contain a point $P$ with $\chi(P) \in \{0,1\}$, and therefore that $\gamma = 4$; 
a more precise discussion is given below.
There do however exist counterexamples for every value of $q$, as is shown by a construction explained in an \texttt{arXiv} version of this paper (\texttt{1605.02162v2}).
\end{remark}

\begin{remark} \label{remarkgon6}
We do not know whether gonality $6$ actually occurs or not. For this one needs to verify the existence of
a non-trigonal genus five curve over $\FF_3$ which is pointless over $\FF_{27}$ and whose discriminant
curve has no $\FF_3$-rational points $P$ for which $\chi(P) \in \{0,1\}$. We ran a naive brute-force
search for such curves, but did not manage to find one.
\end{remark}

If $q$ is large enough and $\mathfrak{D}(\overline{C})$ has at least one (geometrically) irreducible component that is defined over $\FF_q$,
then a point $P \in \mathfrak{D}(\overline{C})(\FF_q)$ with $\chi(P) \in \{0,1\}$ exists and therefore $\overline{C}$ has $\FF_q$-gonality
$4$. To state a precise
bound on $q$, let us analyze the (generic) setting
where $\mathfrak{D}(\overline{C})$ is a non-singular plane quintic. In this case
the `good' points $P$ are in a
natural correspondence with pairs of $\FF_q$-points on an unramified double cover of $\mathfrak{D}(\overline{C})$;
we refer to~\cite[\S2(c)]{beauville} and the references therein
for more background.
By Riemann-Hurwitz this cover is of genus $11$, for which the lower Serre-Weil bound is positive from $q > 467$ on. 
The presence of singularities or of absolutely irreducible $\FF_q$-components of lower degree can be studied in a similar way and leads to smaller bounds. 

There are two possible ways in which $\mathfrak{D}(\overline{C})$ does \emph{not} have an absolutely irreducible $\FF_q$-component:
either it could decompose into two conjugate lines over $\FF_{q^2}$ and three conjugate lines over $\FF_{q^3}$, or it
could decompose into five conjugate lines over $\FF_{q^5}$. But in the former case
the $\FF_q$-rational point $P$ of intersection of the two $\FF_{q^2}$-lines satisfies $\chi(P) = 0$, so here too
our curve $\overline{C}$ has $\FF_q$-gonality $4$. Thus the only remaining case is that of five conjugate lines over $\FF_{q^5}$, which
can occur for every value of $q$.

Let us now address Problem~\ref{liftingproblem}.
First assume that $\gamma = 4$, i.e.\
that there exists a point $P \in \mathfrak{D}(\overline{C})(\FF_q)$ with $\chi(P) \in \{0, 1 \}$. This can be decided quickly: 
if $q \leq 467$ then one can
proceed by exhaustive search, while if $q > 467$ it is sufficient to
verify whether or not $\mathfrak{D}(\overline{C})$ decomposes into five conjugate lines.
To \emph{find} such a point, we first look for $\FF_q$-rational singularities of $\mathfrak{D}(\overline{C})$: these
are exactly the points $P$ for which $\chi(P) = 0$.
If no such singularities exist then we look for a point $P \in \mathfrak{D}(\overline{C})(\FF_q)$ for which $\chi(P) = 1$ 
by trial and error. Once our point has been found, we proceed as follows.


\paragraph*{\underline{$\chi(P) = 0$}}

In this case $P$ corresponds to a rank $3$ quadric, which using a linear change of variables
we can assume to be in the standard form $\overline{S} = ZW - X^2$. Choose homogeneous
\begin{wrapfigure}{r}{5cm}
  \includegraphics[width=4.9cm]{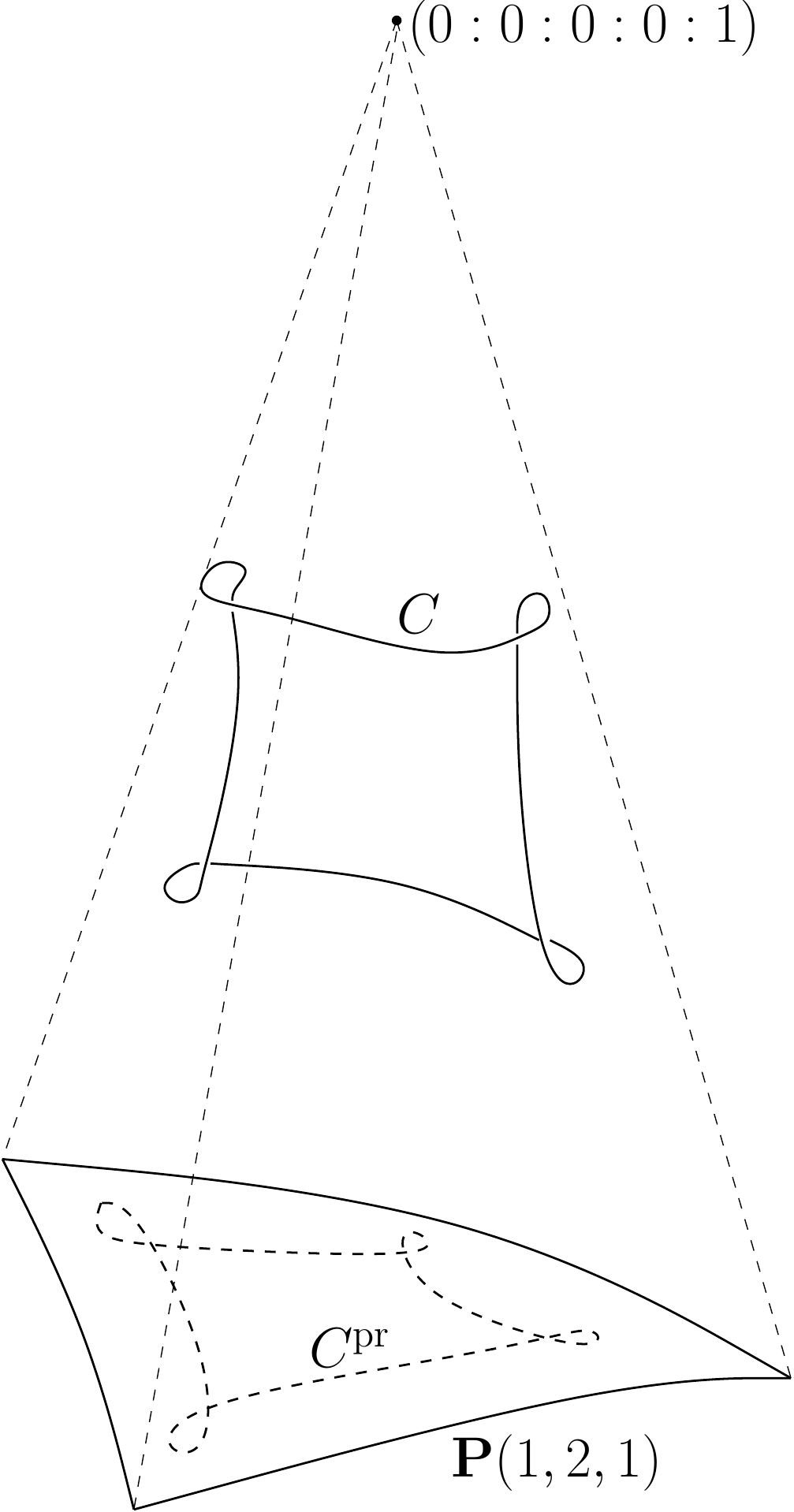}
  \vspace{0.3cm}
\end{wrapfigure}
quadratic polynomials
\[ \overline{S}_2, \overline{S}_2' \in \FF_q[X,Y,Z,W,V] \] 
that along with $\overline{S}$ form a basis of
$\mathcal{I}_2(\overline{C})$. (In practice one can usually take $\overline{S}_2 = \overline{S}_{2,1}$ and $\overline{S}'_2 = \overline{S}_{2,2}$.)
Let $S_2, S_2' \in \mathcal{O}_K[X,Y,Z,W,V]$ be quadrics that reduce to $\overline{S}_2, \overline{S}_2'$ modulo $p$.
Along with 
\[ S = ZW - X^2 \in \mathcal{O}_K[X,Y,Z,W,V] \] 
these cut out a canonical genus $5$ curve $C \subset \PPK^4$. 
We view the quadric defined by $S$ as a cone over the weighted projective plane $\PPK(1,2,1)$
with top $(0:0:0:0:1)$. Our curve is then an intersection of two quadrics inside this cone, and by projecting
from the top we obtain a curve $C^\mathrm{pr}$ in $\PPK(1,2,1)$. In terms of equations this amounts to eliminating
$V$ from $S_2$ and $S_2'$
 by taking the resultant
$S_2^\mathrm{pr} := \text{res}_V (S_2,S_2')$,
which is a homogeneous quartic. Now as in \eqref{genus4conic} we further eliminate the variable $W$ to end up with
$S_2^\mathrm{pr}(XZ,YZ,Z^2,X^2)$.
After dehomogenizing with respect to $Z$, renaming $X \leftarrow x$ and $Y \leftarrow y$ and rescaling if needed, we obtain an affine equation
$ f = y^4 + f_2(x)y^3 + f_4(x)y^2 + f_6(x)y + f_8(x)$,
 with $f_i \in \mathcal{O}_K[x]$ of degree at most $i$. Its Newton polygon
  is contained in (and typically equals):
\begin{center}
\polfig{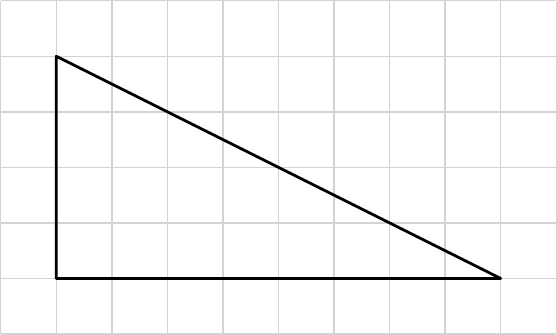}{4}{2.4}{$\Delta_{5,0}^{0}$}
\end{center}
Note that Baker's genus bound reads $9$, so this exceeds the geometric genus by $4$.
Thus it was important to lift $\overline{S}_2, \overline{S}_2'$ before projecting.

 \paragraph*{\underline{$\chi(P) = 1$}}
In this case $P$ corresponds to a rank $4$ quadric whose pseudo-determinant is a square. Using a linear change of variables
we can assume it to be in the standard form $\overline{S} = XY - ZW$, which is a cone over $\PPq^1 \times \PPq^1$ with top $(0:0:0:0:1)$.
Choose homogeneous quadratic polynomials
\[ \overline{S}_2, \overline{S}_2' \in \FF_q[X,Y,Z,W,V] \] 
that along with $\overline{S}$ form a basis of
$\mathcal{I}_2(\overline{C})$. (In practice one can usually take $\overline{S}_2 = \overline{S}_{2,1}$ and $\overline{S}'_2 = \overline{S}_{2,2}$.)
Let $S_2, S_2' \in \mathcal{O}_K[X,Y,Z,W,V]$ be quadrics that reduce to $\overline{S}_2, \overline{S}_2'$ modulo $p$.
Along with 
\[ S = XY - ZW \in \mathcal{O}_K[X,Y,Z,W,V] \]
these cut out a canonical genus $5$ curve $C \subset \PPK^4$, which can be viewed as an intersection of two quadrics inside a cone over $\PPK^1 \times \PPK^1$ with top $(0:0:0:0:1)$. 
We first project from
\begin{wrapfigure}{l}{6cm} 
 \includegraphics[width=5.9cm]{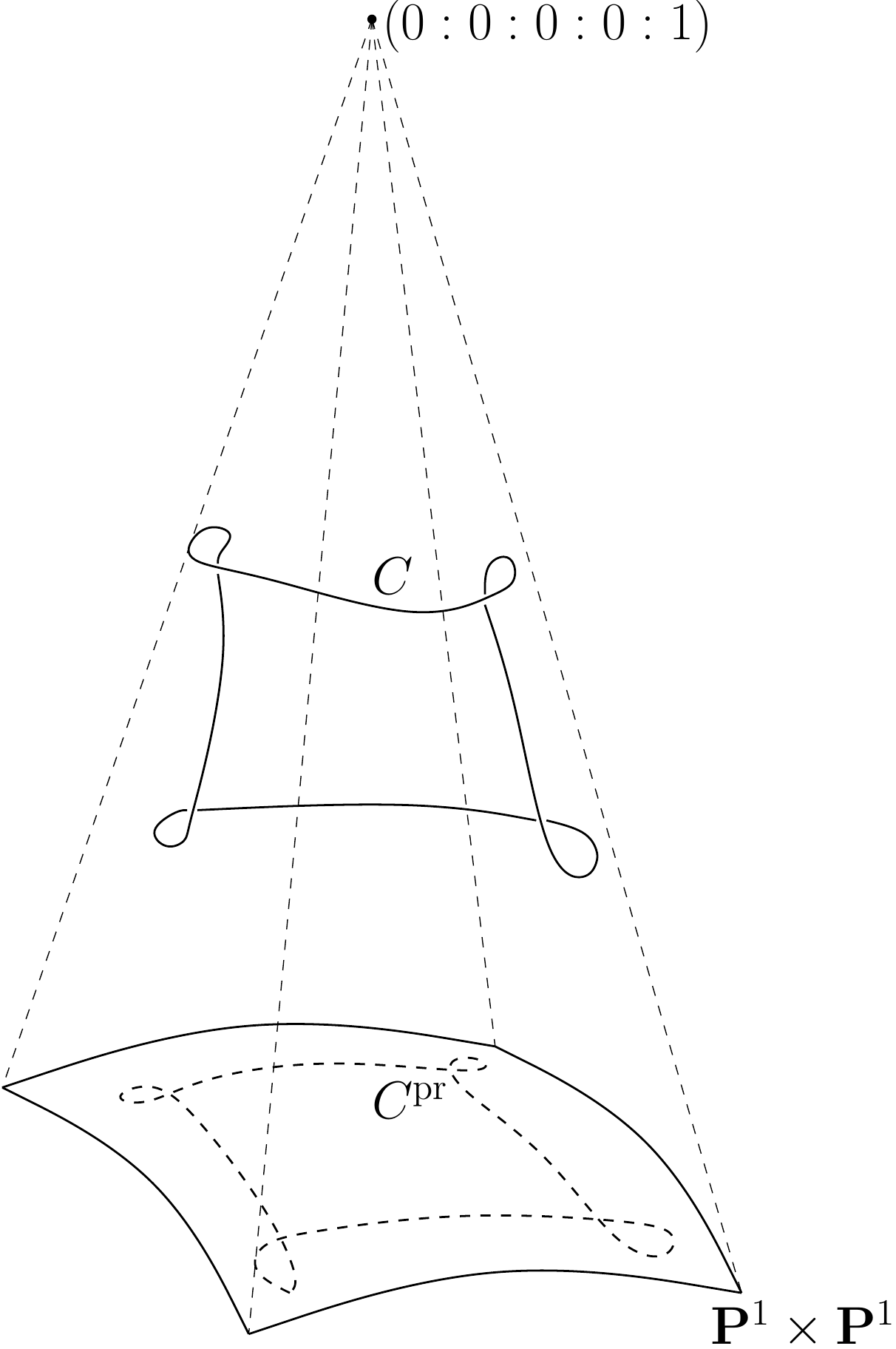}
 \vspace{-1.3cm}
\end{wrapfigure}
this top, to obtain a curve $C^\mathrm{pr}$ in $\PPK^1 \times \PPK^1$.
In terms of equations, this amounts to eliminating $V$ from $S_2$ and $S_2'$ by taking the resultant
$S_2^\mathrm{pr} := \text{res}_V (S_2,S_2')$,
which is a homogeneous quartic. As in the discussion following (\ref{segre}), we conclude
that $C^\mathrm{pr} $ is defined by the bihomogeneous polynomial
\begin{equation} \label{genus5bihomogeneous}
 S_2^\mathrm{pr}(X_0W_0,Y_0Z_0,Z_0W_0,X_0Y_0)
\end{equation}
of bidegree $(4,4)$. Let $f \in \mathcal{O}_K[x,y]$ be the polynomial
obtained from (\ref{genus5bihomogeneous}) by dehomogenizing with respect to $Z_0$ and $W_0$ and
by renaming $X_0 \leftarrow x$ and $Y_0 \leftarrow y$.
Then the Newton polygon of $f$ is contained in (and typically equals):
\begin{center}
\polfig{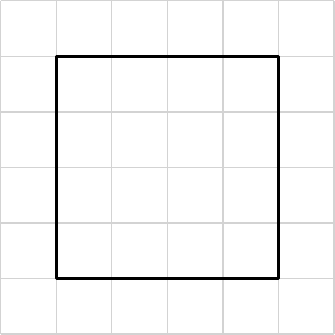}{2.4}{2.4}{$\Delta_{5,1}^0$}
\end{center}
In particular $\deg_y f = 4$, as wanted.
Here again Baker's bound reads $9$, which exceeds the geometric genus by $4$.

 \paragraph*{\underline{$\forall P \in \mathfrak{D}(\overline{C})(\FF_q): \chi(P) = -1$}} This case is very rare, so we will be
 rather sketchy here.
 If $\gamma = 6$ then we do not know how to address Problem~\ref{liftingproblem}, 
  which for point counting purposes is not an issue because this could only occur when $q = 3$. 
  If $\gamma = 5$ then one can try to address Problem~\ref{liftingproblem} 
by following the proof of Lemma~\ref{genus5gonality}, similar to the way we treated the $
\chi(\det \overline{M}_2) = -1$ case in genus four. For instance this works as follows if $\overline{C}(\FF_q)$ has
at least three non-collinear points, which is guaranteed as soon as $\#\overline{C}(\FF_q) \geq 4$, which
in turn is guaranteed if $q > 101$ by the Serre-Weil bound. Apply a transformation of $\PPq^4$ to position
these points at $(0:1:0:0:0)$, $(0:0:0:1:0)$ and $(0:0:0:0:1)$, so that the plane they span is $X = Z = 0$. This implies that
the defining quadrics have no terms in $Y^2$, $W^2$ and $V^2$, a property which is of course easily preserved
when lifting to $\mathcal{O}_K[X,Y,Z,W,V]$, resulting in a curve $C \subset \PPK^4$ again passing through
$(0:1:0:0:0)$, $(0:0:0:1:0)$ and $(0:0:0:0:1)$. Eliminating $W$ and $V$, which geometrically amounts to projecting 
from the line $X = Y = Z = 0$, results in a sextic in $\PPK^2 = \proj K[X,Y,Z]$ passing through $(0:1:0)$ in a non-singular way (otherwise
the pencil of lines through that point would cut out a $K$-rational $g^1_4$). We can therefore
apply a projective transformation that maps the corresponding tangent line to infinity, while keeping the point at $(0:1:0)$.
Then by dehomogenizing with respect to $Z$ and
renaming $X \leftarrow x$ and $Y \leftarrow y$
we end up with a polynomial $f \in \mathcal{O}_K[x,y]$ whose Newton polygon is contained in (and typically equals): 
\begin{center}
  \polfig{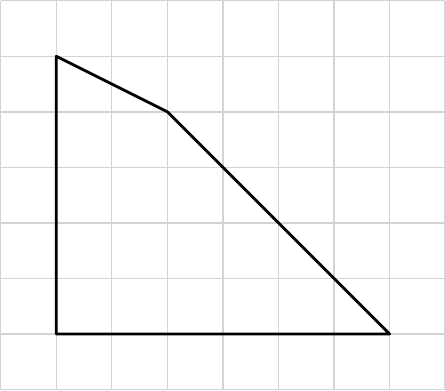}{3.2}{2.8}{$\Delta^5_5$}
\end{center}
We omit a further discussion.\\

\noindent \varhrulefill[0.4mm]
\vspace{-0.3cm}

\begin{algo} \label{algorithm_genus5}
Lifting curves of genus $5$: basic solution 

\vspace{-0.2cm}
\noindent \varhrulefill[0.4mm]

\noindent \textbf{Input:} non-hyperelliptic genus $5$ curve $\overline{C}/\FF_q$ of $\FF_q$-gonality $\gamma \leq 5$

\noindent \qquad \qquad \qquad \qquad \qquad \qquad \qquad \qquad \qquad \quad or of $\FF_q$-gonality $\gamma = 5$ and $\# \overline{C}(\FF_q) \geq 4$

\noindent \textbf{Output:} lift $f \in \mathcal{O}_K[x,y]$ satisfying (i), (ii), (iii) that is supported

\noindent \qquad \qquad \qquad $\bullet$ on $\Delta_{5,\text{trig}}^{0,0}$ if $\overline{C}$ is trigonal, or else

\noindent \qquad \qquad \qquad $\bullet$ on $\Delta_{5,0}^0$ if $\exists P \in \mathfrak{D}(\overline{C}): \chi(P) = 0$, or else

\noindent \qquad \qquad \qquad $\bullet$ on $\Delta_{5,1}^0$ if $\exists P \in \mathfrak{D}(\overline{C}): \chi(P) = 1$, or else

\noindent \qquad \qquad \qquad $\bullet$ on $\Delta_5^5$

\vspace{-0.2cm}
\noindent \varhrulefill[0.4mm]

\noindent \small \phantom{0}1 \normalsize: $\overline{C} \gets \text{CanonicalImage}(\overline{C})$ in $\PPq^4 = \proj \FF_q[X,Y,Z,W,V]$

\noindent \small \phantom{0}2 \normalsize: \textbf{if} $\text{Ideal}(\overline{C})$ is generated by quadrics \textbf{then}

\noindent \small \phantom{0}3 \normalsize: \quad $\overline{S}_{2,1}, \overline{S}_{2,2}, \overline{S}_{2,3} \gets \text{quadrics that generate $\text{Ideal}(\overline{C})$}$

\noindent \small \phantom{0}4 \normalsize: \quad $\overline{M}_i \gets \text{Matrix}(\overline{S}_{2,i})$ ($i=1,2,3$)

\noindent \small \phantom{0}5 \normalsize: \quad $\mathfrak{D}(\overline{C}) \gets$ curve in $\PPq^2 = \proj \FF_q[\lambda_1, \lambda_2, \lambda_3]$ defined by $\det(\lambda_1 \overline{M}_1 + \lambda_2 \overline{M}_2 + \lambda_3 \overline{M}_3)$

\noindent \small \phantom{0}6 \normalsize: \quad \textbf{if} $q \leq 467$ and $\forall P \in \mathfrak{D}(\overline{C})(\FF_q): \chi(P) = -1 $ (verified exhaustively)

\noindent \small \phantom{0}7 \normalsize: \quad \quad \textbf{or} $q > 467$ and $\mathfrak{D}(\overline{C})$ decomposes into five conjugate lines \textbf{then}

\noindent \small \phantom{0}8 \normalsize: \quad \quad goodpoints $\gets$ false 

\noindent \small \phantom{0}9 \normalsize: \quad \textbf{else} 

\noindent \small 10 \normalsize: \quad \quad goodpoints $\gets$ true

\noindent \small 11 \normalsize: \quad \textbf{if} goodpoints \textbf{then}

\noindent \small 12 \normalsize: \quad \quad \textbf{if} $\mathfrak{D}(\overline{C})$ has $\FF_q$-rational singular point $P$ \textbf{then}

\noindent \small 13 \normalsize: \quad \quad \quad $\overline{S}_2, \overline{S}_2' \gets \text{quadrics such that } \langle 
    \overline{S}_P, \overline{S}_2, \overline{S}_2' \rangle_{\FF_q} = \langle \overline{S}_{2,1}, \overline{S}_{2,2}, \overline{S}_{2,3} 
    \rangle_{\FF_q}$

\noindent \small 14 \normalsize: \quad \quad \quad apply automorphism of $\PPq^4$ transforming $\overline{S}_P$ into $WZ - X^2$

\noindent \small 15 \normalsize: \quad \quad \quad $S_2 \gets \text{NaiveLift}(\overline{S}_2)$; $S_2' \gets \text{NaiveLift}(\overline{S}_2')$; $S_2^\text{pr} \gets \text{res}_V(S_2, S_2')$

\noindent \small 16 \normalsize: \quad \quad \quad \textbf{return} Dehomogenization${}_Z(S_2^\text{pr}(XZ,YZ,Z^2,X^2))$

\noindent \small 17 \normalsize: \quad \quad \textbf{else}

\noindent \small 18 \normalsize: \quad \quad \quad \textbf{repeat} $P \gets \text{Random}(\mathfrak{D}(\overline{C})(\FF_q))$ \textbf{until} $\chi(P) = 1$

\noindent \small 19 \normalsize: \quad \quad \quad $\overline{S}_2, \overline{S}_2' \gets \text{quadrics such that } \langle 
    \overline{S}_P, \overline{S}_2, \overline{S}_2' \rangle_{\FF_q} = \langle \overline{S}_{2,1}, \overline{S}_{2,2}, \overline{S}_{2,3} 
    \rangle_{\FF_q}$

\noindent \small 20 \normalsize: \quad \quad \quad apply automorphism of $\PPq^4$ transforming $\overline{S}_P$ into $XY - ZW$

\noindent \small 21 \normalsize: \quad \quad \quad $S_2 \gets \text{NaiveLift}(\overline{S}_2)$; $S_2' \gets \text{NaiveLift}(\overline{S}_2')$; $S_2^\text{pr} \gets \text{res}_V(S_2, S_2')$

\noindent \small 22 \normalsize: \quad \quad \quad \textbf{return} Dehomogenization${}_Z(S_2^\text{pr}(XZ,YZ,Z^2,XY))$

\noindent \small 23 \normalsize: \quad \textbf{else} 

\noindent \small 24 \normalsize: \quad \quad $P_1, P_2, P_3 \leftarrow$ distinct random points of $\overline{C}(\FF_q)$

\noindent \small 25 \normalsize: \quad \quad apply automorphism of $\PPq^4$ sending $P_1$, $P_2$, $P_3$ to $(0:1:0:0:0)$, 

\noindent \small \phantom{25} \normalsize \hfill $(0:0:0:1:0)$, $(0:0:0:0:1)$

\noindent \small 26 \normalsize: \quad \quad $S_{2,i} \leftarrow \text{NaiveLift}(\overline{S}_{2,i})$ $(i = 1,2,3)$

\noindent \small 27 \normalsize: \quad \quad $C^\text{pr} \leftarrow \text{res}_{W,V}(S_{2,1},S_{2,2},S_{2,3})$

\noindent \small 28 \normalsize: \quad \quad apply automorphism of $\PPK^2$ transforming $T_{(0:1:0)}(C^\text{pr})$ into $Z=0$

\noindent \small 29 \normalsize: \quad \quad \textbf{return} Dehomogenization${}_Z(C^\text{pr})$

\noindent \small 30 \normalsize: \textbf{else}

\noindent \small 31 \normalsize: \quad apply automorphism of $\PPq^4$ transforming space of quadrics in $\text{Ideal}(\overline{C})$ to

\noindent \small \phantom{31} \normalsize \hfill $\langle X^2 - ZV, XY - ZW, XW - YV \rangle_{\FF_q}$ 

\noindent \small 32 \normalsize: \quad $\overline{S}_{3,1}, \overline{S}_{3,2} \gets \text{cubics that along with quadrics generate $\text{Ideal}(\overline{C})$}$

\noindent \small 33 \normalsize: \quad $\overline{f}_i \gets \text{Dehomogenization}_{Z}(\overline{S}_{3,i}(XZ,YZ,Z^2,X^2,XY))$ ($i=1,2$)

\noindent \small 34 \normalsize: \quad \textbf{return} NaiveLift($\gcd(\overline{f}_1, \overline{f}_2)$)

\vspace{-0.2cm}
\noindent \varhrulefill[0.4mm]
\end{algo}

\subsubsection{Optimizations} \label{optim_genus5}

\paragraph*{Trigonal case}
By applying (\ref{mademonic}) to 
a polynomial with Newton polygon $\Delta_{5,\text{trig}}^{0,0}$
we end up with a polynomial $f \in \mathcal{O}_K[x,y]$ that is monic in $y$ and that has degree $5 + (\gamma - 1)2 = 9$ in $x$. This
can be improved as soon as our curve $\overline{C} / \FF_q$ has a rational point $P$, which is guaranteed if $q > 89$ by the Serre-Weil bound (probably this bound is not optimal).
The treatment below is very similar to the genus four case where $\chi_2(\det \overline{M}_2) = 0$, as elaborated in Section~\ref{optim_genus4}. 
The role of $\PPq(1,2,1)$ is now played by our scroll $\overline{S}(1,2)$. Recall that the latter
is a ruled surface spanned by a line (the directrix) and a conic that are being parameterized simultaneously.
Using an automorphism of $\overline{S}(1,2)$ we can position
$P$ at the point at infinity of the spanning conic, in such a way that the curve and the conic meet at $P$ with multiplicity at least 
two. 
This results in a Newton polygon that is contained in (and typically equals):
\begin{center}
\polfig{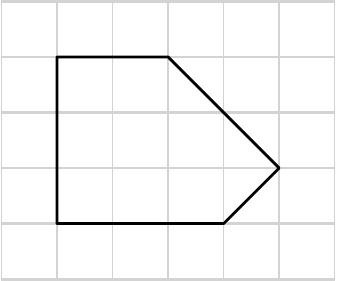}{2.4}{2}{$\Delta_{5,\text{trig}}^{0,1}$}
\end{center}
See Remark~\ref{autsofS12} below for how this can be done in practice.
Here an application of (\ref{mademonic}) typically results in $\deg_x f = 3 + (\gamma - 1)2 = 7$. There are two caveats here: our curve might exceptionally 
be tangent at $P$ to a rule of the scroll, in which case it is impossible to make it tangent to the conic at that point. Or worse: our point $P$ might lie on the directrix, in which case it is just impossible to move it to the spanning conic. 
In these cases one can most likely just retry with another $P$.
But in fact these two situations are better, as explained in Remark~\ref{genus5trigonalcaveatremark} below.

   \begin{remark}~\label{autsofS12}
   The automorphisms of $\overline{S}(1,2)$ can be applied directly to $\overline{f}$. They correspond to
   \begin{itemize}
     \item substituting 
   $y \leftarrow \overline{a} y + \overline{b} x + \overline{c}$ and $x \leftarrow \overline{a}' x + \overline{b}'$ in $
  \overline{f}$ for some 
  $\overline{a},\overline{a}' \in \FF_q^\ast$ and $\overline{b},\overline{b}',\overline{c} \in \FF_q$,
    \item exchanging the rule at infinity for the $y$-axis by replacing $\overline{f}$ by $x^5 \overline{f}(x^{-1},x^{-1}y)$, 
  \end{itemize} or to
  a composition of both. 
  For instance imagine that an affine point $P = (\overline{a},\overline{b})$ was found with a non-vertical
  tangent line. Then $\overline{f} \leftarrow \overline{f}(x + \overline{a}, y + \overline{b})$ translates this point to the origin,
  at which the tangent line becomes of the form $y = \overline{c} x$. Substituting $\overline{f} \leftarrow \overline{f}(x,y + \overline{c}x)$ positions this line horizontally, and finally replacing $\overline{f}$ by $x^5 \overline{f}(x^{-1},x^{-1}y)$ results
  in a polynomial with Newton polygon contained in $\Delta_{5,\text{trig}}^{0,1}$.
  \end{remark}
  
\begin{remark}[non-generic optimizations] \label{genus5trigonalcaveatremark}
As for the first caveat, if $\overline{C}$ turns out to be tangent at $P$ to one of the rules of the scroll then moving $P$ to the point at infinity of the 
spanning conic results in a Newton polygon that is contained in (and typically equals):
\begin{center}
\polfig{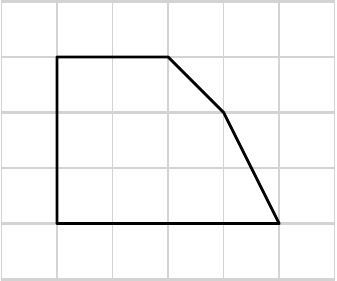}{2.4}{2}{$\Delta_{5,\text{trig}}^{0,2}$}
\end{center}
Even though this yields $\deg_x f = 4 + (\gamma - 1)2 = 8$, the corresponding point count is slightly faster. Such a $P$ will 
exist if and only if the ramification scheme of $(x,y) \mapsto x$ has an $\FF_q$-rational point. Following the 
heuristics from Remark~\ref{genus3flexremark}
we expect that this works in about $1 - 1/e$ of the cases.
As for the second caveat, if $P$ is a point on the directrix of the scroll, we can move it to its point at infinity. 
This results
in a Newton polygon that is contained in (and typically equals) the left polygon below.
\begin{center}
\polfig{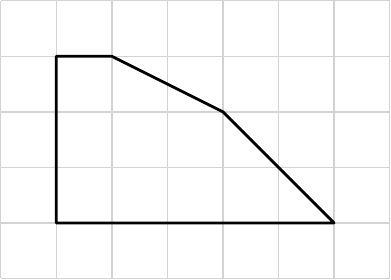}{2.8}{2}{$\Delta_{5,\text{trig}}^{1,0}$}
\polfig{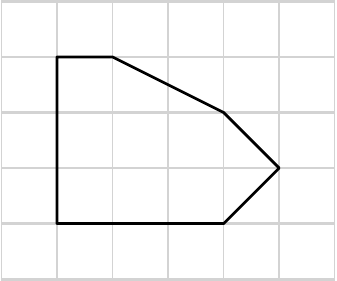}{2.4}{2}{$\Delta_{5,\text{trig}}^{1,1}$}
\end{center}
This again gives us $\deg_x f = 5 + (\gamma - 1)1 = 7$, but here too the corresponding point count is faster. As explained in an \texttt{arXiv}
version of our paper (\texttt{1605.02162v2}), the probability of being able to realize this polygon is about $1/2$, and one
can even end up inside the right polygon with a probability of about $3/8$, yielding
$\deg_x f = 4 + (\gamma - 1)1 = 6$.
\end{remark}

%


\paragraph*{Non-trigonal case} 
For point counting purposes
it is advantageous to give preference 
to the case $\chi(P) = 0$, i.e.\ to use a singular 
point $P \in \mathfrak{D}(\overline{C})(\FF_q)$ if
it exists. 
Some optimizations over the corresponding discussion in Section~\ref{optim_genus5} are possible,
for instance generically one can replace $\Delta_{5,0}^0$ with the left polygon below:
\begin{center}
\polfig{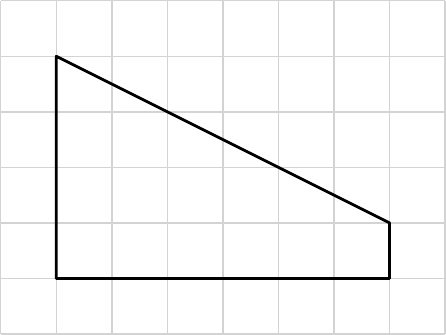}{3.2}{2.4}{$\Delta_{5,0}^{1}$} \quad
\polfig{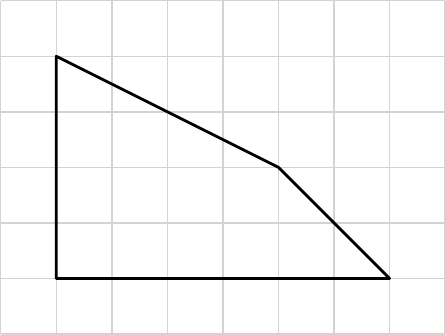}{3.2}{2.4}{$\Delta_{5,0}^{2}$}
\end{center}
With an estimated probability of about $1 - (3/8)^\rho$ one can even end up inside the right polygon.
Here $10 \geq \rho \geq 1$ denotes the number of singular points $P \in \mathfrak{D}(\overline{C})(\FF_q)$. 
We will spend a few more words on this in Remark~\ref{genus5conicaloptimizationremark} below, after having
discussed the $\chi(P) = 1$ case.
However usually such a singular $\FF_q$-point $P$ does not exist, i.e.\ $\rho = 0$. More precisely we expect that 
the proportion of curves for which $\mathfrak{D}(\overline{C})$
is a smooth plane quintic tends to $1$ as $q \rightarrow \infty$. Indeed, in terms of moduli the locus
of (non-hyperelliptic, non-trigonal) genus five curves having a singular point on its discriminant curve has codimension one;
see \cite{teixidor,looijenga}${}^\dagger$. For this reason
we will focus our attention on the case $\chi(P) = 1$, and leave it to the interested
reader to elaborate the remaining details.


As for the case $\chi(P) = 1$, note that by applying (\ref{mademonic}) to 
a polynomial with Newton polygon $\Delta_{5,1}^{0}$
one ends up with a polynomial that is monic in $y$ and that has degree $4 + (\gamma - 1)4 = 16$ in $x$.
With near certainty this can be reduced to $10$, as we will explain now. The idea is to exploit the fact that in practice 
the discriminant curve $\mathfrak{D}(\overline{C})$
contains enough $\FF_q$-rational points for there to be considerable freedom in choosing a $P$ for which 
$\chi(P) = 1$. We want to select a suited such $P$, by which we mean the following.

As before, assume that an automorphism of $\PPq^4$ has been applied
such that $\overline{S}_P = \overline{S} = XY - ZW$ and let $\overline{S}_2, \ \overline{S}_2' \in \FF_q[X,Y,Z,W,V]$
be quadrics that along with $\overline{S}$ cut out our curve $\overline{C}$. 
Now suppose that we would have projected $\overline{C}$
 from the point $(0:0:0:0:1)$ \emph{before} lifting to characteristic $0$. Then we would have ended up with a curve $\overline{C}^\text{pr}$
 in \[ \PPq^1 \times \PPq^1 : \overline{S} = 0 \quad \text{in } \PPq^3 = \proj \FF_q[X,Y,Z,W]. \]
 This curve has
arithmetic genus $9$, because in fact that is what Baker's bound measures. Since the excess in genus is $9 - 5 = 4$
we typically expect there to be $4$ nodes. Our point $P$ is `suited'
as soon as one of the singular points $Q$ of $\overline{C}^\text{pr}$ is $\FF_q$-rational. If $P$ is not suited, i.e.\ if
there is no such $\FF_q$-rational
singularity, then we retry with another $P \in \mathfrak{D}(\overline{C})(\FF_q)$ for which $\chi(P) = 1$.
Heuristically we estimate the probability of success to be about $5/8$. In particular if there are enough candidates for $P$ available,
 we should end up being successful very quickly with overwhelming probability.

Given such a singular point $Q \in \overline{C}^\text{pr}(\FF_q) \subset \PPq^1 \times \PPq^1$ we can move it to the point 
$((1:0),(1:0))$, similar to what we did in the genus $4$ case where $\chi_2(\det \overline{M}_2) = 1$. In terms of the
coordinates $X,Y,Z,W$ of the ambient space $\PPq^3$ this means moving the point to $(0:0:0:1)$. 
Let's say this amounts to the change of variables
\[ \begin{pmatrix} X \\ Y \\ Z \\ W \\ \end{pmatrix} \leftarrow A \cdot \begin{pmatrix} X \\ Y \\ Z \\ W \\ \end{pmatrix} \]
where $A \in \FF_q^{4 \times 4}$. Then we can apply the change of variables
\[ \begin{pmatrix} X \\ Y \\ Z \\ W \\ V \\ \end{pmatrix} \leftarrow \begin{pmatrix} A & 0 \\ 0 & 1 \\ \end{pmatrix} \cdot \begin{pmatrix} X \\ Y \\ Z \\ W \\ V \\ \end{pmatrix} \]
directly to the defining polynomials $\overline{S}, \overline{S}_1, \overline{S}_2$ of $\overline{C}$ to obtain the curve $\overline{C}_\text{tr}$ cut out by
\[
  \overline{S} = XY - ZW, \ \overline{S}_{2,\text{tr}}, \ \overline{S}_{2',\text{tr}} \in \FF_q[X,Y,Z,W,V].
\]
Indeed the transformation affects $\overline{S}$ at most through multiplication by a non-zero scalar. If we would now project from $(0:0:0:0:1)$ as before,
we would end up with a curve $\overline{C}_\text{tr}^\text{pr} \subset \PPq^1 \times \PPq^1$ having a singularity at $((1:0),(1:0))$, which is at $(0:0:0:1)$ in the coordinates $X,Y,Z,W$. 

Recall that inside $\PPq^4$ we view $\overline{S}$ as the defining equation of a cone over $\PPq^1 \times \PPq^1$ with top $(0:0:0:0:1)$. The fact that the projected curve has a singularity at $(0:0:0:1)$ 
implies that
the line $X = Y = Z = 0$ meets the curve at least twice, counting multiplicities (these points of intersection need not be $\FF_q$-rational). Thus after multiplying $\overline{S}_{2,\text{tr}}$ by a scalar if needed we find 
that 
\[ \overline{S}_{2,\text{tr}}(0,0,0,W,V) = \overline{S}'_{2,\text{tr}}(0,0,0,W,V) =  \overline{a}W^2 + \overline{b}WV + \overline{c}V^2 \]
for some $\overline{a}, \overline{b}, \overline{c} \in \FF_q$.
Now lift $\overline{S}_{2,\text{tr}}$ and $\overline{S}_{2',\text{tr}}$ in a consistent way, in order 
to obtain quadrics $S_2, S_2' \in \mathcal{O}_K[X,Y,Z,W,V]$ satisfying
\begin{equation*}
  S_2(0,0,0,W,V)  =  S_2'(0,0,0,W,V)  = aW^2 + bWV + cV^2
\end{equation*}
for elements $a,b,c \in \mathcal{O}_K$ that reduce to $\overline{a}, \overline{b}, \overline{c}$ 
modulo $p$.
If we then proceed as before, we end up with a curve $C^\text{pr}$ in $\PPK^1 \times \PPK^1$ 
having a singularity at $((1:0),(1:0))$.
This eventually results in a defining polynomial $f \in \mathcal{O}_K[x,y]$ whose Newton polygon is contained in (and typically equals): 
\begin{center}
\polfig{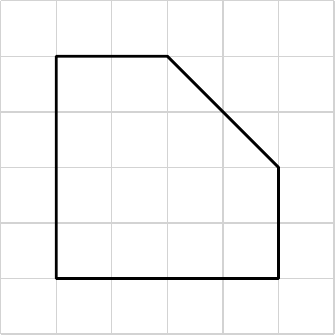}{2.4}{2.4}{$\Delta_{5,1}^2$}
\end{center}
Applying \eqref{mademonic} to $f$ results in a polynomial having degree at most $4 + (\gamma - 1)2 = 10$ in $x$, as announced.

\begin{remark} \label{genus5conicaloptimizationremark}
The same ideas apply to the case $\chi(P) = 0$, with the role of $\PPq^1 \times \PPq^1$ replaced by $\PPq(1,2,1)$.
If the projection $\overline{C}^\text{pr}$ of $\overline{C}$ to $\PPq(1,2,1)$ has an $\FF_q$-rational singular point, then it can be arranged
that the resulting curve $C^\text{pr} \subset \PPK(1,2,1)$ has a singularity at $(1:0:0)$, eventually yielding 
a polynomial $f \in \mathcal{O}_K[x,y]$ whose Newton polygon is contained in $\Delta^2_{5,0}$. As 
in the $\chi(P) = 1$ case we expect that
the probability that this works out for a given $P$ is about $5/8$. But
unlike the $\chi(P) = 1$ case there is not much freedom to retry in the case of failure: we have $\rho$ chances only.
This explains our expected probability of $1 - (3/8)^\rho$ to be able to realize $\Delta^2_{5,0}$. 

 If the foregoing fails every time then
we can play the same game with a non-singular $\FF_q$-rational point $Q$ on $\overline{C}^\text{pr}$ (guaranteed
to exist if $q > 89$ because then $\overline{C}$ has an $\FF_q$-rational point by the Serre-Weil bound). The
result is a curve $C^\text{pr} \subset \PPK(1,2,1)$ containing the point $(1:0:0)$. We can then use an automorphism of
$\PPK(1,2,1)$ to make $C^\text{pr}$ tangent to $X=0$ at that point (unless the tangent line is vertical, in which case we simply retry with another $Q$). This is done similarly to the way we handled the case $\chi_2(\det \overline{M}_2) = 0$ in Section~\ref{optim_genus4}: see in particular Remark~\ref{autsofP121}.
In this way one ends up in $\Delta^1_{5,0}$.
\end{remark}

\subsubsection{Implementation}

The tables below contain timings, memory usage and failure rates for the trigonal and non-trigonal case and various values of $p$ and $q=p^n$. 
For the precise meaning of the various entries in the tables see Section~\ref{section_genus3implementationandtimings}.\\


\vspace{-0.2cm}

\noindent \textbf{Trigonal}\\

\noindent \scriptsize
\tabcolsep=0.11cm
\begin{tabular}{r||r|r|r|r}
             & time     &time    & space  & fails   \\
$p$          & lift(s)  &pcc(s)  & (Mb)   & /1000   \\
\hline \hline
$11$         & $0.02$   & $0.6$  & $96$   & $206$   \\
$67$         & $0.02$   & $2.4$  & $96$   & $45$    \\
$521$        & $0.02$   & $23$   & $112$  & $4$     \\
$4099$       & $0.02$   &$358$   &$548$   & $1$     \\
$32771$      & $0.02$   &$4977$  &$3982$  & $0$    
\end{tabular}
\quad
\begin{tabular}{r||r|r|r|r}
             & time     &time    & space  & fails  \\
$q$          & lift(s)  &pcc(s)  & (Mb)   & /1000  \\
\hline \hline
$3^5$        & $0.1$    & $17$  & $108$   & $6$    \\
$7^5$        & $0.1$    & $33$  & $150$   & $0$    \\
$17^5$       & $0.2$    & $76$  &$556$    & $0$    \\
$37^5$       & $0.2$    & $186$ &$1070$   & $0$    \\
$79^5$       & $0.3$    & $452$ &$1716$   & $0$   
\end{tabular}
\quad
\begin{tabular}{r||r|r|r|r}
             & time   &time     & space  & fails  \\
$q$          & lift(s)&pcc(s)   & (Mb)   & /1000  \\
\hline \hline
$3^{10}$     & $1.2$  &  $82$   & $188$  & $0$    \\
$7^{10}$     & $2.0$  & $214$   & $621$  & $0$    \\
$17^{10}$    & $3.6$  & $587$   &$1366$  & $0$    \\
$37^{10}$    & $4.5$  &$1584$   &$2453$  & $0$    \\
$79^{10}$    & $6.3$  &$4039$   &$4176$  & $0$   
\end{tabular}\\


\normalsize

\vspace{0.3cm}

\noindent \textbf{Non-trigonal}\\

\noindent \scriptsize
\tabcolsep=0.11cm
\begin{tabular}{r||r|r|r|r}
             & time     &time    & space  & fails   \\
$p$          & lift(s)  &pcc(s)  & (Mb)   & /1000   \\
\hline \hline
$11$         &  $0.1$   & $2.0$  & $64$   & $14$    \\
$67$         &  $0.1$   & $7.2$  & $76$   & $0$     \\
$521$        &  $0.2$   &$65$    & $165$  & $0$     \\
$4099$       &  $0.2$   &$1326$  &$1326$  & $0$     \\
$32771$      &  $0.2$   &$21974$ &$10329$ & $0$    
\end{tabular}
\quad
\begin{tabular}{r||r|r|r|r}
             & time     &time    & space   & fails  \\
$q$          & lift(s)  &pcc(s)  & (Mb)    & /1000  \\
\hline \hline
$3^5$        & $2.5$    & $59$   &  $229$  & $0$    \\
$7^5$        & $5.3$    & $114$  &  $352$  & $0$    \\
$17^5$       & $10$     & $261$  &  $556$  & $0$    \\
$37^5$       & $14$     & $662$  &  $919$  & $0$    \\
$79^5$       & $19$     & $1552$ &  $1494$ & $0$   
\end{tabular}
\quad
\begin{tabular}{r||r|r|r|r}
             & time   &time      & space  & fails  \\
$q$          & lift(s)&pcc(s)    & (Mb)   & /1000  \\
\hline \hline
$3^{10}$     & $16$   &$504$     & $780$  & $0$   \\
$7^{10}$     & $40$   &$1191$    & $1304$ & $0$   \\
$17^{10}$    & $89$   &$2946$    & $2231$ & $0$   \\
$37^{10}$    & $128$  &$7032$    & $3679$ & $0$   \\ 
$79^{10}$    & $193$  &$15729$   & $6267$ & $0$     
\end{tabular}

\bigskip
\normalsize





\section{Curves of low gonality} \label{section_lowgonality}

\subsection{Trigonal curves} \label{section_trigonal}

Recall
from Remark~\ref{remark_gonalityoverFq} that from genus five on
a curve $\overline{C} / \FF_q$ is trigonal iff it is geometrically trigonal.
It is known~\cite{saintdonat} that a minimal set of generators for the ideal of a canonical model 
$\overline{C} \subset \PPq^{g-1} = \proj \FF_q[X_1,X_2,\dots,X_g] $
of a non-hyperelliptic curve of genus $g \geq 4$ over $\FF_q$ consists of
\begin{itemize}
  \item $(g-2)(g-3)/2$ quadrics 
  \[ \overline{S}_{2,1}, \overline{S}_{2,2}, \dots, \overline{S}_{2,(g-2)(g-3)/2} \] 
  and $g-3$ cubics 
  \[ \overline{S}_{3,1}, \overline{S}_{3,2}, \dots, \overline{S}_{3,g-3} \] 
  if $\overline{C}$ is trigonal or $\FF_q$-isomorphic to a smooth curve in $\PPq^2$ of degree five,
  \item just $(g-2)(g-3)/2$ quadrics in the other cases. 
\end{itemize}
So given such a minimal set of generators, it is straightforward to decide trigonality, unless $g=6$ in which case
one might want to check whether $\overline{C}$ is isomorphic to a smooth plane quintic or not. See Remark~\ref{veronese} below for how to do this.

From now on assume that we are given a trigonal curve $\overline{C} / \FF_q$ in the above canonical form.
Then the quadrics $\overline{S}_{2,i}$ spanning $\mathcal{I}_2(\overline{C})$ are known to
define a rational normal surface scroll $\overline{S}$ of type $(a,b)$,
where $a,b$ are non-negative integers satisfying 
\begin{equation} \label{maroniconditions}
 a \leq b, \qquad a + b = g - 2, \qquad b \leq (2g-2)/3,
\end{equation}
called the Maroni invariants\footnote{The existing literature is ambiguous on this terminology. Some authors talk about \emph{the} Maroni invariant of a trigonal curve, in which case they could mean either $a = \min(a,b)$, or $b-a$.} of $\overline{C}$. This means that up to a linear change of variables, 
it is the image $\overline{S}(a,b)$ of
\[ \PPq^1 \times \PPq^1 \hookrightarrow \PPq^{g-1} : ((s:t),(u:v)) \mapsto (ut^a : ut^{a-1}s :  \dots : us^a : vt^b :  vt^{b-1}s :  \dots : vs^b), \]
i.e.\ it is the ruled
surface obtained by simultaneously parameterizing 
\begin{itemize}
\item
a rational normal curve of degree $a$ in the $\PPq^a$ corresponding to 
$X_1, X_2, \dots, X_{a+1}$, and
\item
a rational normal curve of degree $b$ in the $\PPq^b$ corresponding to $X'_1, X'_2, \dots, X'_{b+1}$,
where $X'_i$ denotes the variable $X_{a+1+i}$,
\end{itemize}
each time drawing the rule through the points under consideration (each of these rules
intersects our trigonal curve in three points, counting multiplicities). 

As a consequence, modulo a linear
change of variables, the space $\mathcal{I}_2(\overline{C})$ admits the $2 \times 2$ minors of
\begin{equation} \label{generalscrolleqs}
  \left( \begin{array}{cccc} X_1 & X_2 & \dots & X_a \\ X_2 & X_3 & \dots & X_{a + 1}  \\ \end{array} \right| \hspace{-0.1cm}
  \left. \begin{array}{cccc} X'_1 & X'_2 & \dots & X'_b \\ X'_2 & X'_3 & \dots & X'_{b+1}  \\ \end{array} \right)
\end{equation}
as a basis, for some $a,b$ satisfying \eqref{maroniconditions}. We assume that we have a function \texttt{ConvertScroll} at our disposal that
upon input of $\mathcal{I}_2(\overline{C})$ and a pair $(a,b)$ satisfying \eqref{maroniconditions}, either
\emph{finds} such a linear change of variables, or outputs `wrong type' in case the surface cut out by $\mathcal{I}_2(\overline{C})$ is
not a scroll of type $(a,b)$.

\begin{remark} \label{blinduse}
If $g=5$ then $(1,2)$ is the only pair of integers satisfying \eqref{maroniconditions}, and one can use our ad hoc method from mentioned in Section~\ref{section_genus5lifting} to find 
the requested linear change of variables as above. For higher genus we have written an experimental version of \texttt{ConvertScroll} in Magma, 
which can be found in the file \texttt{convertscroll.m}. It blindly relies 
on Schicho's function \texttt{ParametrizeScroll}, which implements
the Lie algebra method from~\cite{GHPS}. Unfortunately the latter is only guaranteed to  
 work in characteristic zero, and indeed one runs into trouble when naively applying \texttt{ParametrizeScroll}
 over finite fields of very small characteristic; empirically however, we found that $p > g$ suffices for a slight modification of \texttt{ParametrizeScroll} to work consistently.
We remark that it is an easy linear algebra problem to verify the correctness of the output, in case it is returned. In any case further research is needed to turn this into a more rigorous step.
\end{remark}

\begin{remark} \label{maroniorder}
If `wrong type' is returned then one retries with another pair $(a,b)$ satisfying \eqref{maroniconditions}.
From a moduli theoretic point of view~\cite{stohr}${}^\dagger$ the most likely case is $a = b = (g-2)/2$ if $g$ is even,
and $a + 1 = b = (g-1)/2$ if $g$ is odd, so it is wise to try that pair first, and then to let $a$ decrease gradually.
According to~\cite{schichosevilla}${}^\dagger$ the Lie algebra
method implicitly computes the Maroni invariants, so it should in fact be possible to get rid of this trial-and-error part; recall that we just use the function \texttt{ConvertScroll} as a black box.
\end{remark}

\begin{remark}[$g=6$] \label{veronese}
If `wrong type' is returned on input $(2,2)$ as well as on input $(1,3)$, then we are in the smooth plane quintic case and
therefore $\overline{C}$ is not trigonal. Here $\mathcal{I}_2(\overline{C})$ cuts out a Veronese surface in $\PPq^5$,
rather than a scroll. We will revisit this case at the end of the section. 
\end{remark}

Once our quadrics $\overline{S}_{2,i}$ are given by the minors of \eqref{generalscrolleqs}, we restrict our curve $\overline{C}$ to the embedded torus 
\[ \TTq^2 \hookrightarrow \PPq^{g-1} : (x,y) \mapsto (y: xy: \dots : x^ay : 1 : x : \dots : x^b) \]
by simply substituting 
\[ X_1 \leftarrow y, \ X_2 \leftarrow xy, \ \dots, \ X_{a+1} \leftarrow x^ay \quad \text{and} \quad X'_1 \leftarrow 1, \ X'_2 \leftarrow x, \ \dots, \ X'_{b+1} \leftarrow x^b. \]
This makes the quadrics vanish identically, while the cubics become
\[ \overline{f}_1, \overline{f}_2, \dots, \overline{f}_{g-3} \in \FF_q[x,y]. \]
The ideal generated by these polynomials is principal, i.e.\ of the form $(\overline{f})$, where the Newton polygon of $\overline{f} = \gcd(\overline{f}_1, \overline{f}_2, \dots, \overline{f}_{g-3})$ is contained in (and typically equals):
\begin{center}
\begin{minipage}[b]{6.4cm}
\begin{center}
  \includegraphics[height=2.7cm]{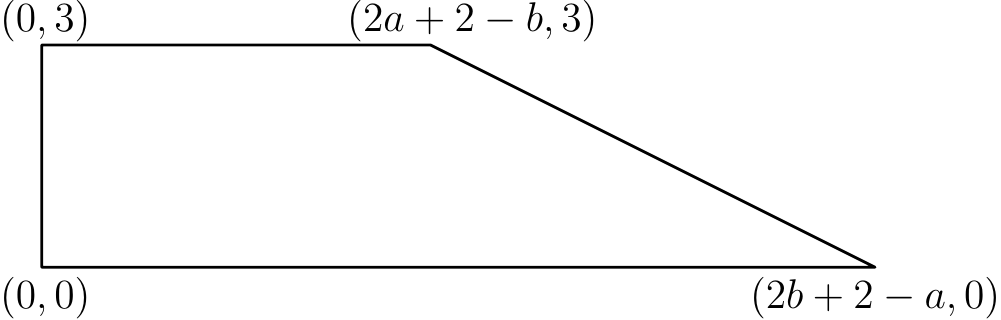}
\end{center}
\end{minipage}
\end{center}
The correctness of these claims follows for instance from~\cite[\S3]{cacocanonical}. Note that in particular $\overline{f}$ attains Baker's bound, so a naive Newton polygon preserving lift $f \in \mathcal{O}_K[x,y]$ satisfies (i), (ii) and (iii).

\begin{remark} 
 It should be clear that the above is a generalization of the corresponding method from Section~\ref{section_genus5lifting}, where we dealt with trigonal curves of genus five. 
 But the method also generalizes the genus four cases $\chi_2(\det \overline{M}_2) = 0$ and $\chi_2(\det \overline{M}_2) = 1$ from Section~\ref{section_genus4lifting}, where the scrolls are $\overline{S}(0,2) = \PPq(1,2,1)$ and $\overline{S}(1,1) = \PPq^1 \times \PPq^1$, respectively.
\end{remark}

\begin{remark}
Here too one could try to compress the Newton polygon by clipping off boundary points, similar to what we did in Section~\ref{optim_genus5}.
But as the genus grows the resulting speed-ups become less and less significant, and we omit a further discussion.
\end{remark}

\paragraph*{Example}
Let us carry out the foregoing procedure for the curve defined by
\[ (x^3 + x + 1)y^3 + 42(2x^4 + x^3 + 3x^2 + 3x + 1)y^2 + (x+1)(x^4 + 2x^2 + x + 1)y + 42(x^2+1) = 0 \]
over $\FF_{43}$. This is the reduction mod $43$ of 
the modular curve $X_0^+(164)$, or rather an affine model of it, whose equation we took from~\cite{modular}.  It is of genus $6$, while we note that Baker's bound reads $7$,
so it is not met here. Using the intrinsic~\texttt{CanonicalMap}
one computes that

\vspace{-0.2cm}
\scriptsize
\[ \left\{ \begin{array}{l}  X_1^2X_2 + 42X_1^2X_5 + 40X_1^2X_6 + 40X_1X_2X_6 + X_1X_3^2 + 2X_1X_3X_6 + 
        42X_1X_4^2 + 40X_1X_4X_5 + X_1X_4X_6 \\ \ \qquad \qquad + 6X_1X_5X_6 + 7X_1X_6^2 + 
        42X_2X_3^2 + 2X_2X_3X_6 + 41X_2X_6^2 + 42X_3^3 + 40X_3X_6^2 + 2X_4^2X_5 
        + 4X_4^2X_6 \\ \ \qquad \qquad + 4X_4X_5X_6 + X_4X_6^2 + 38X_5X_6^2 + 39X_6^3 \\
    X_1^2X_3 + 42X_1^2X_6 + 39X_1X_2X_6 + X_1X_3^2 + 38X_1X_3X_6 + 42X_1X_4X_5 + 
        X_1X_5X_6 + 7X_1X_6^2 + X_2X_3^2 \\ \ \qquad \qquad + 41X_2X_3X_6 + 8X_2X_6^2 + 42X_3^2X_6 + 
        4X_3X_6^2 + X_4^2X_6 + 5X_4X_5X_6 + X_4X_6^2 + 40X_5X_6^2 + 37X_6^3 \\
    42X_1^2X_6 + X_1X_2X_3 + 42X_1X_2X_6 + 39X_1X_3X_6 + 42X_1X_4X_5 + 
        42X_1X_5X_6 + 6X_1X_6^2 + X_2X_3^2 + 39X_2X_3X_6 \\ \ \qquad \qquad + 7X_2X_6^2 + X_3^3 + 
        42X_3^2X_6 + 5X_3X_6^2 + 42X_4^2X_6 + 5X_4X_5X_6 + 41X_4X_6^2 + X_5X_6^2 
        + 36X_6^3 \\
    42 X_1X_3 + 42X_1X_5 + X_2^2 + X_2X_6 + X_3X_6 + 42X_4^2 + 42X_4X_6 + X_5X_6 \\
    42X_1X_5 + X_2X_4 + X_2X_6 + 42X_4^2 + 42X_4X_6 + X_5X_6 \\
    42X_1X_6 + X_3X_4 + X_3X_6 + 42X_4X_5 + X_6^2 \\
    42X_1X_6 + X_2X_5 + 42X_4X_5 + X_6^2 \\
    42X_2X_6 + X_3X_5 \\
    42X_4X_6 + X_5^2 + 42X_6^2 \\ \end{array}  \right. \]
\normalsize is a minimal set of generators for the ideal $\mathcal{I}(\overline{C})$ of 
a canonical model $\overline{C} \subset \PPq^5$. We are clearly in the trigonal case, so the six quadrics must cut out a rational normal surface scroll. According to \eqref{maroniconditions} the type of the latter is either $(1,3)$ or $(2,2)$.
Following Remark~\ref{maroniorder} we first try $(2,2)$, so we search for a linear change of variables
taking $\mathcal{I}_2(\overline{C})$ to the space of quadrics spanned by the $2 \times 2$ minors of
\begin{equation*} 
  \left( \begin{array}{cc} X_1 & X_2 \\ X_2 & X_3 \\ \end{array} \right| \hspace{-0.1cm}
  \left. \begin{array}{cc} X_4 & X_5 \\ X_5 & X_6 \\ \end{array} \right).
\end{equation*}
Our experimental version of the function \texttt{ConvertScroll} turns out to work here, and the type $(2,2)$ was a correct guess: the change of variables returned by Magma reads
\[ \begin{pmatrix} X_1 \\ X_2 \\ X_3 \\ X_4 \\ X_5 \\ X_6 \\ \end{pmatrix} \leftarrow 
\begin{pmatrix} 40 & 3 & 42 & 0 & 30 & 33 \\
    0 & 12 & 35 & 40 & 42 & 2 \\
    0 & 9 & 4 & 30 & 29 & 42 \\
    20 & 37 & 5 & 2 & 8 & 22 \\
    22 & 19 & 11 & 28 & 32 & 14 \\
    38 & 29 & 16 & 21 & 33 & 36 \\
                 \end{pmatrix} \cdot \begin{pmatrix} X_1 \\ X_2 \\ X_3 \\ X_4 \\ X_5 \\ X_6 \\ \end{pmatrix}. \]
Applying this transformation to our generators of $\mathcal{I}(\overline{C})$ and then substituting 
\[ X_1 \leftarrow y, \ \ X_2 \leftarrow xy, \ \ X_3 \leftarrow x^2y, \ \ X_4 \leftarrow 1, \ \ X_5 \leftarrow x, \ \ X_6 \leftarrow x^2 \]
annihilates the quadrics, while the cubics become
\[ 6(x+27)(x+32)\overline{f}, \ \ 39(x+13)(x+20)\overline{f}, \ \ 2(x+13)^2\overline{f}\] 
respectively, where
\[ \begin{array}{rcl} \overline{f} \hspace{-0.2cm} & = & \hspace{-0.2cm} x^4y^3 + 8x^4y^2 + 31x^4y + 29x^4 + 37x^3y^3 + 23x^3y^2 + 
    16x^3y + x^3 + 12x^2y^3 + 18x^2y^2  \\
  &  & \qquad \qquad \quad + 12x^2y + 25x^2 + 10xy^3 + 
    7xy^2 + 30xy + 11x + 13y^3 + 36y^2 + 3y + 2. \\ \end{array} \]
For this polynomial Baker's bound is attained, so a naive lift to $f \in \mathcal{O}_K[x,y]$ satisfies (i), (ii), (iii).
After making $f$ monic using \eqref{mademonic} it can be fed to the algorithm from \cite{tuitman1,tuitman2} to find
the numerator 
\begin{multline*}
43^6 T^{12} + 43^5 \cdot 8 T^{11} + 43^4 \cdot 154 T^{10} + 43^3 \cdot 1032 T^9 + 43^2 \cdot 9911 T^8 + 43 \cdot 62496 T^7 \\
 + 444940 T^6 + 62496 T^5 + 9911 T^4 + 1032 T^3 + 154 T^2 + 8 T + 1
\end{multline*}
of the zeta function $Z_{\overline{C} / \FF_{43}}(T)$ in a couple of seconds.

\paragraph*{Point counting timings}
Despite the lack of a well-working function \texttt{ConvertScroll}, we can 
tell how the point counting algorithm from~\cite{tuitman1,tuitman2} should perform in composition with the above method,
by simply assuming that $\overline{C}$ is \emph{given} as the genus $g$ curve defined by a suitably generic polynomial $\overline{f} \in \FF_q[x,y]$ 
supported on $\conv \{ (0,0), (2b + 2 - a,0), (2a + 2 - b, 3), (0,3) \}$. Then we
can immediately lift to $\mathcal{O}_K[x,y]$. The tables below give point counting timings and memory usage for 
randomly chosen such polynomials in genera $g = 6,7$, where for the sake of conciseness we
restrict to the generic Maroni invariants $a = \lfloor (g - 2)/2 \rfloor$ and $b = \lceil (g-2)/2 \rceil$; the other Maroni invariants
give rise to faster point counts.\\


\vspace{-0.2cm}

\noindent \textbf{$\mathbf{g=6}$}\\

\noindent \scriptsize
\tabcolsep=0.11cm
\begin{tabular}{r||r|r}
             & time    & space   \\
$p$          & pcc(s)  & (Mb)    \\
\hline \hline
$11$         & $0.9$   & $32$    \\
$67$         & $6.0$   & $32$    \\
$521$        & $70$    & $118$   \\
$4099$       & $769$   & $824$   \\
$32771$      & $8863$  &$6829$      
\end{tabular}
\quad
\begin{tabular}{r||r|r}
             & time      & space   \\
$q$          & pcc(s)    & (Mb)    \\
\hline \hline
$3^5$        &  $33$     & $76$    \\
$7^5$        &  $64$     & $80$    \\
$17^5$       &  $176$    &$197$    \\
$37^5$       &  $415$    &$371$    \\
$79^5$       &  $1035$   &$791$   
\end{tabular}
\quad
\begin{tabular}{r||r|r}
             &time     & space   \\
$q$          &pcc(s)   & (Mb)    \\
\hline \hline
$3^{10}$     & $183$   & $188$  \\
$7^{10}$     &$503$    &$320$   \\
$17^{10}$    &$1490$   &$749$   \\
$37^{10}$    &$3970$   &$1663$  \\
$79^{10}$    &$10945$  &$3716$ 
\end{tabular}\\


\normalsize
\vspace{0.3cm}

\noindent \textbf{$\mathbf{g=7}$}\\

\noindent \scriptsize
\tabcolsep=0.11cm
\begin{tabular}{r||r|r|r}
             & time    & space  & \\
$p$          & pcc(s)  & (Mb)   & \\
\hline \hline
$11$         & $1.5$   & $32$   & \\
$67$         & $6.5$   & $32$   & \\
$521$        & $88$    & $118$  & \\
$4099$       & $955$   & $857$  & \\
$32771$      &$13279$  & $6983$ &    
\end{tabular}
\quad
\begin{tabular}{r||r|r|r}
             & time      & space  & \\
$q$          & pcc(s)    & (Mb)   & \\
\hline \hline
$3^5$        &  $43$     & $76$   & \\
$7^5$        &  $91$     & $118$  & \\
$17^5$       &  $257$    &$241$   & \\
$37^5$       &  $602$    &$460$   & \\
$79^5$       &  $1561$   &$983$   &
\end{tabular}
\quad
\begin{tabular}{r||r|r|r}
             &time     & space  & \\
$q$          &pcc(s)   & (Mb)   & \\
\hline \hline
$3^{10}$     & $283$   &$197$  & \\
$7^{10}$     &$777$    &$371$  & \\
$17^{10}$    &$2384$   &$919$  & \\
$37^{10}$    &$6706$   &$2212$ & \\
$79^{10}$    &$18321$  &$4682$ &
\end{tabular}

\normalsize

\paragraph*{Smooth plane quintics}
We end this section with a brief discussion of the genus $6$ case where our canonical curve $\overline{C} \subset \PPq^5$ is $\FF_q$-isomorphic to a smooth plane quintic.
Such curves are never trigonal: using a variant of Lemma~\ref{genus3gonality} one verifies
that the $\FF_q$-gonality is $4$ if and only if $\# \overline{C}(\FF_q) > 0$, which is guaranteed if $q > 137$
by the Serre-Weil bound. In the other cases it is $5$. 
Nevertheless from the point of view of the canonical embedding, smooth plane quintics behave `as if they were trigonal', which is why we include them here. (The appropriate unifying statement reads that trigonal curves and smooth plane quintics
are exactly the curves having Clifford index $1$.)
Here our main task towards tackling Problem~\ref{liftingproblem} is to find a linear change of variables transforming
the space $\mathcal{I}_2(\overline{C})$ into 
\[ \langle X_2^2 - X_1X_4,
    X_2X_3 - X_1X_5,
    X_3^2 - X_1X_6,
    X_3X_4 - X_2X_5,
    X_3X_5 - X_2X_6,
    X_5^2 - X_4X_6 \rangle_{\FF_q} \]
whose zero locus is the Veronese surface in `standard form', i.e.\ the closure of the image of 
\[  \TTq^2 \hookrightarrow \PPq^5 : (x, y) \mapsto (x^2 : xy : x : y^2 : y : 1). \]
In order to achieve this, we simply assume that we have a function \texttt{ConvertVeronese} at our disposal. One could
again try to use Schicho's function \texttt{ParametrizeScroll} for this, but here too we expect problems because of the characteristic being finite (although we did not carry out
the experiment).
Once this standard form is attained, an easy substitution
\[ X_1 \leftarrow x^2, \ X_2 \leftarrow xy, \ X_3 \leftarrow x, \ X_4 \leftarrow y^2, \ X_5 \leftarrow y, \ X_6 \leftarrow 1 \]
makes the quadrics vanish identically, while the cubics have a gcd whose homogenization defines the desired smooth plane quintic.
From here one proceeds as in the smooth plane quartic case described in Section~\ref{basicsolutiontogenus3}.

\subsection{Tetragonal curves} \label{section_tetragonal}

We conclude this article with some thoughts on how the foregoing material can be adapted to the tetragonal case.
A full elaboration of the steps below (or even a rigorous verification of some corresponding claims) lies beyond our current scope. 
In particular we have not implemented anything of what follows. The main aim of this section is twofold: to illustrate how our treatment of non-trigonal curves of genus five from Section~\ref{section_genus5lifting} naturally fits within a larger framework,
and to propose a track for future research, involving mathematics that was developed mainly by Schreyer in~\cite[\S6]{schreyer}${}^\dagger$ and Schicho, Schreyer and Weimann in~\cite[\S5]{weimann}. 

Let 
\[ \overline{C} \subset \PPq^{g-1} = \proj \overline{R}, \qquad \overline{R} = \FF_q[X_1,X_2,\dots,X_g] \]
be the canonical model of a genus $g \geq 5$ curve that is non-hyperelliptic, non-trigonal, and
not isomorphic to a smooth plane quintic, so that
a minimal set of generators of $\mathcal{I}(\overline{C}) \subset \overline{R}$
consists of $\beta_{12} := (g-2)(g-3)/2$ quadrics
\[ \overline{S}_{2,1}, \overline{S}_{2,2}, \dots, \overline{S}_{2,\beta_{12}}. \]
The notation $\beta_{12}$ refers to the corresponding entry in the graded Betti table
of the homogeneous coordinate ring of $\overline{C}$, to which we will make a brief reference at the
end of this section. Assume that the $\FF_q$-gonality of $\overline{C}$ is four, and consider a corresponding
$\FF_q$-rational
\begin{wrapfigure}{r}{6.1cm} 
  \hfill \includegraphics[width=6cm]{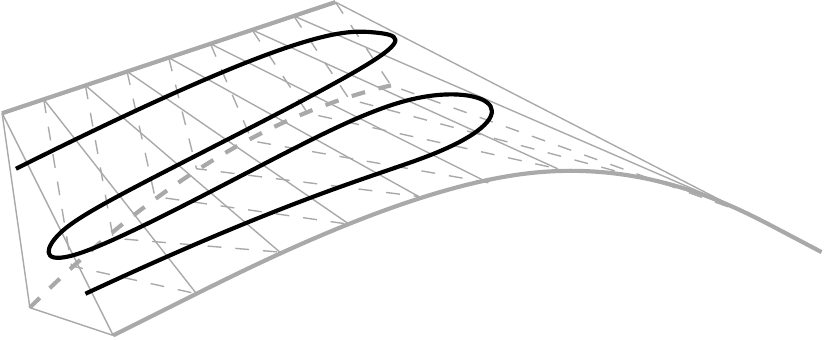} 
\end{wrapfigure}
map $\pi : \overline{C} \rightarrow \PPq^1$. 
We note that unlike the trigonal case this map may
not
be uniquely determined modulo
   automorphisms of $\PPq^1$,
even for $g$ arbitrarily large. 
The linear spans of the   
 fibers of $\pi$ form a one-dimensional family of
planes in $\PPq^{g-1}$ that cut out a rational normal \emph{threefold} scroll $\overline{S}$.
Similar to before, up to a linear change of variables, such a scroll is obtained by simultaneously
parameterizing
\begin{itemize}
\item a rational normal curve of degree $a$ in the $\PPq^a$ corresponding to 
$X_1, X_2, \dots, X_{a+1}$, 
\item
a rational normal curve of degree $b$ in the $\PPq^b$ corresponding to $X'_1, X'_2, \dots, X'_{b+1}$,
where $X'_i$ denotes the variable $X_{a+1+i}$, and
\item
a rational normal curve of degree $c$ in the $\PPq^c$ corresponding to $X''_1, X''_2, \dots, X''_{c+1}$,
where $X''_i$ denotes the variable $X_{a+b+2+i}$,
\end{itemize}
each time taking the plane connecting the points under consideration (each of these planes intersects our trigonal curve in four points, counting multiplicities). Again this concerns a determinantal variety, defined by
the $2 \times 2$ minors of
\begin{equation} \label{tetragonalscrolleqs}
  \left( \begin{array}{cccc} X_1 & X_2 & \dots & X_a \\ X_2 & X_3 & \dots & X_{a + 1}  \\ \end{array} \right| \hspace{-0.1cm}
  \left. \begin{array}{cccc} X'_1 & X'_2 & \dots & X'_b \\ X'_2 & X'_3 & \dots & X'_{b+1}  \\ \end{array} \right. \hspace{-0.1cm}
  \left| \begin{array}{cccc} X''_1 & X''_2 & \dots & X''_c \\ X''_2 & X''_3 & \dots & X''_{c + 1}  \\ \end{array} \right). 
\end{equation}
 Alternatively our scroll can be thought of as the Zariski closure of the image of
\[ \TTq^3 \hookrightarrow \PPq^{g-1} : (x,y,z) \mapsto (z:xz: \dots : x^az : y : xy : \dots : x^by : 1 : x : \dots : x^c ), \]
or if one prefers, as the toric threefold associated to the polytope
\begin{center}
\begin{minipage}[b]{6.4cm}
\begin{center}
  \includegraphics[height=2.4cm]{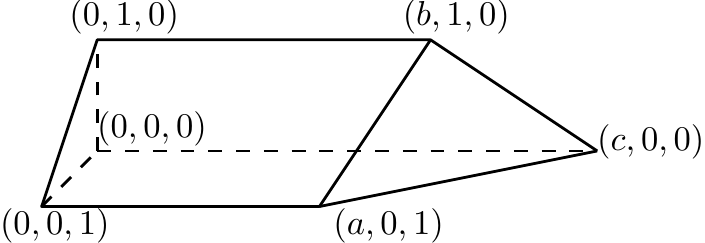}
\end{center}
\end{minipage}
\qquad
\begin{minipage}[b]{2cm}
$(\Delta_{(a,b,c)}).$
\vspace{0.9cm}
\end{minipage}
\end{center}
Let us denote this `standard' scroll in $\PPq^{g-1}$ by $\overline{S}(a,b,c)$. 
The non-negative integers $(a,b,c)$ are called the scrollar invariants of $\overline{C}$ with respect to $\pi$ and can be chosen to satisfy 
\begin{equation} \label{scrollarconditions}
 a \leq b \leq c, \qquad a + b + c = g - 3, \qquad c \leq (2g-2)/4,
\end{equation} 
where the last inequality follows from Riemann-Roch.

Inside the scroll $\overline{S}$ our curve $\overline{C}$ arises as a complete intersection of two hypersurfaces $\overline{Y}$ and $\overline{Z}$ that are `quadratic'. More precisely the Picard group of $\overline{S}$ is generated by the class $[ \overline{H} ]$ of a hyperplane section and the class $[\overline{\Pi}]$ of a ruling (i.e.\ of the linear span of a fiber of $\pi$), and $\overline{Y}$ and $\overline{Z}$ can be chosen such that
\[ \overline{Y} \in 2[ \overline{H} ] - b_1 [\overline{\Pi}] , \qquad \overline{Z} \in 2[ \overline{H} ] - b_2 [\overline{\Pi}] \]
for non-negative integers $b_1 \geq b_2$ satisfying $b_1 + b_2 = g-5$. 
These integers are invariants of the curve, that is, they do not depend on the choice of $\pi$. If $b_2 < b_1$ then
also the surface $\overline{Y}$ is uniquely determined by $\overline{C}$. This is automatic when $g$ is even. 

Let us now assume that $\overline{S}$ is given in the standard form $\overline{S}(a,b,c)$, which we consider along with the embedded torus $\TTq^3$. Then for $\overline{Y}$
to be in the class $2[ \overline{H} ] - b_1 [\overline{\Pi}]$ it means that $\overline{Y} \cap \TTq^3$ is defined by an irreducible polynomial $\overline{f}_{\overline{Y}} \in \FF_q[x,y,z]$
whose support is contained in
\begin{center}
\begin{minipage}[b]{5.8cm}
\begin{center}
  \includegraphics[height=2.4cm]{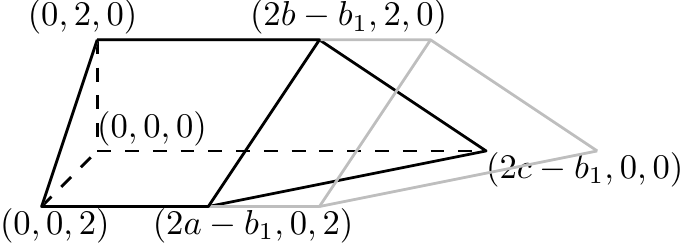}
\end{center}
\end{minipage}
\qquad
\begin{minipage}[b]{2cm}
$(\Delta_{(a,b,c),b_1}).$ 
\vspace{0.9cm}
\end{minipage}
\end{center}
or more precisely\footnote{Indeed, the coordinate $2a - b_1$ might be negative; an example of such behaviour can be found in an \texttt{arXiv} version of this paper (\texttt{1605.02162v2}).} in
\[ \conv \{ (0,0,0), (2c-b_1,0,0), (0,2,0), (2b-b_1,2,0), (0,0,2), (2a-b_1,0,2) \} \cap \RR_{\geq 0}^3. \]
In other words this is the polytope obtained from $2 \Delta_{(a,b,c)}$ by shifting its right-most face leftwards over a distance $b_1$. 
Moreover $b_1$ is the maximal integer for which this containment holds.
The same applies to $\overline{Z}$, leading to a polynomial $\overline{f}_{\overline{Z}} \in \FF_q[x,y,z]$ whose support is contained in 
$\Delta_{(a,b,c),b_2}$, which is the polytope obtained from $2 \Delta_{a,b,c}$ by shifting the right-most face inwards over a distance $b_2$.


The main observation of this section is that $\overline{f}_{\overline{Y}}, \overline{f}_{\overline{Z}} \in \FF_q[x,y,z]$ 
is a pair of polynomials meeting a version of Baker's bound for complete intersections, again due to 
Khovanskii~\cite{khovanskiicomplete}${}^\dagger$. In the case of two trivariate polynomials supported on polytopes $\Delta_1$ and $\Delta_2$
the bound reads
\[ g \leq \# \left( \text{interior points of $\Delta_1 + \Delta_2$} \right) - \# \left( \text{interior points of $\Delta_1$} \right) - \# \left( \text{interior points of $\Delta_2$} \right). \]
In our case where $\Delta_1 = \Delta_{(a,b,c),b_1}$ and $\Delta_2 = \Delta_{(a,b,c),b_2}$, this indeed evaluates to $g - 0 - 0 = g$.
Thus the strategy would be similar: lift these polynomials in a Newton polytope preserving way
to polynomials $f_Y, f_Z \in \mathcal{O}_K[x,y,z]$. These
then again cut out a
genus $g$ curve in $\TTK^3$, and
a polynomial $f \in \mathcal{O}_K[x,y]$ satisfying (i)-(iii) can be found by taking the resultant of $f_Y$ and $f_Z$ with respect to $z$ (or with respect to $y$).

\paragraph*{Genus $5$ curves revisited} Let us revisit our treatment of tetragonal curves of genus five 
$\overline{C} \subset \PPq^4 = \proj \FF_q[X,Y,Z,W,V]$ from Section~\ref{section_genus5lifting}. 
\begin{enumerate}
\item Our first step was
to look for a point $P \in \mathfrak{D}(\overline{C})(\FF_q)$ for which $\chi(P) = 0$ or $\chi(P) = 1$.
The corresponding quadrics were described as cones over $\PPq(1,2,1)$ and $\PPq^1 \times \PPq^1$, respectively.
But in the current language these are just rational normal threefold scrolls of type $(0,0,2)$ resp.\ $(0,1,1)$.
Note that this shows that the scroll $\overline{S}$ may indeed depend on the choice of $\pi$.

\item For ease of exposition let us restrict to the case $\chi(P) = 1$. Then the second step was to transform the quadric into 
$XY - ZW$, whose zero locus is the Zariski closure of
\[ \TTq^3 \hookrightarrow \PPq^4 : (x,y,z) \mapsto (1 : xy : x : y : z), \]
i.e.\ the transformation takes the scroll $\overline{S}(0,1,1)$ into `standard form'. 

\item The other quadrics $\overline{S}_2, \overline{S}_2'$ are instances of the surfaces $\overline{Y}$ and $\overline{Z}$.
They are both in the class $2[\overline{H}]$, i.e.\ $b_1 = b_2 = 0$.
Viewing $\overline{Y}$ and $\overline{Z}$ inside the torus $\TTq^3$ amounts to evaluating them at $(1,xy,x,y,z)$,
resulting in polynomials that are supported on
\begin{center}
\begin{minipage}[b]{4cm}
\begin{center}
  \includegraphics[height=2.2cm]{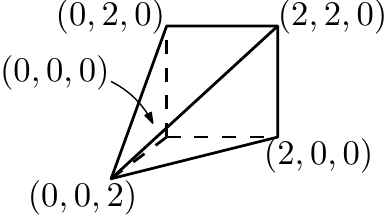}
\end{center}
\end{minipage}
\end{center}
as predicted. With the present approach we naively lift these polynomials to $f_Y, f_Z \in \mathcal{O}_K[x,y,z]$.
In Section~\ref{section_genus5lifting} we applied this naive lift directly to $\overline{S}_2, \overline{S}_2'$, which 
was fine there, but in higher genus it is more convenient to work in $\TTq^3$, since $\overline{Y}, \overline{Z} \subset \overline{S}$ will no longer be cut out by a single quadratic hypersurface of $\PPq^{g-1}$.

\item The last step was to project this lifted curve from $(0:0:0:0:1)$, which in our case amounts to taking
the resultant of $f_Y, f_Z$ with respect to $z$.

\end{enumerate}

\paragraph*{General recipe}
If we want to turn the above into a rigorous recipe for lifting tetragonal curves, three questions show up naturally.
We share some brief first thoughts, but further research is needed regarding each of these. 

\begin{enumerate}
 \item How do we decide whether the input curve has $\FF_q$-gonality $4$ or not, and how
 do we extract from $\mathcal{I}_2(\overline{C})$ the equations of a corresponding rational normal threefold scroll $\overline{S}$?
 
 In genus five we used the discriminant curve for this, but in general the desired information should be traceable from (the first few steps of) a minimal free resolution 
 \[ \overline{R}(-4)^{\beta_{34}} \oplus \overline{R}(-5)^{ \beta_{35} } \rightarrow \overline{R}(-3)^{\beta_{23}} \oplus \overline{R}(-4)^{\beta_{24}} \rightarrow \overline{R}(-2)^{\beta_{12}} \rightarrow \overline{R} \rightarrow \faktor{\overline{R}}{(\overline{S}_{2,1}, \dots, \overline{S}_{2, \beta_{12}})} \]
 of the homogeneous coordinate ring of $\overline{C}$ as a graded $\overline{R}$-module, 
 thanks to a proven part of Green's canonical syzygy conjecture~\cite[Thm.\,2.5]{weimann}, namely 
 that $\beta_{24} \neq 0$ if and only if $\overline{C}$ is $\overline{\FF}_q$-tetragonal or $\FF_q$-isomorphic
 to a smooth plane sextic, which in turn holds if and only if $\overline{C}$ has Clifford index $2$. (The dimensions $\beta_{ij}$ are
 usually gathered in the so-called graded Betti table of $\overline{C}$, and in general Green's conjecture predicts that the Clifford index
 equals the number of leading zeroes on the cubic strand, i.e.\ the minimal $i$ for which $\beta_{i,i+2} \neq 0$.)
 
 If $g \geq 7$ then a sufficiently generic geometrically tetragonal curve satisfies $\beta_{24} = g-4$. This is what Schicho, Schreyer and Weimann~\cite[Ex.\,4.2]{weimann} refer to as the \emph{goneric} case; see also~\cite[Thm.\,0.3]{farkaskemeny}${}^\dagger$. It implies that our curve admits a unique $g^1_4$, hence it is $\FF_q$-tetragonal, and that the ideal of the corresponding scroll $\overline{S}$ can be computed as the annihilator
 of the cokernel of the map \[ \overline{R}(-5)^{ \beta_{35}} \rightarrow \overline{R}(-4)^{ \beta_{24}}. \] See~\cite[Prop.\,4.11]{weimann}. 
 
 In the non-goneric cases one has
 $\beta_{24} = (g-1)(g-4)/2$ and a finer analysis is needed. Some further useful statements 
 can be found in~\cite{weimann} and~\cite{harrison}${}^\dagger$.
  
 \item How do we find the type $(a,b,c)$ of the scroll $\overline{S}$, along with a linear change of variables taking it
  into the standard form $\overline{S}(a,b,c)$ cut out by the minors of \eqref{tetragonalscrolleqs}?
 
 We encountered an analogous hurdle in the trigonal case. Here too it would be natural to try the Lie algebra method from~\cite{GHPS}, 
 but as mentioned this was designed to work over fields of characteristic zero, and it is not clear to us how easily the method carries over to small finite characteristic.
 
 \item How do we find the invariants $b_1, b_2$ along with hypersurfaces $\overline{Y} \in 2[\overline{H}] - b_1[\overline{\Pi}]$ and $\overline{Z} \in 2[\overline{H}] - b_2[\overline{\Pi}]$ that inside $\overline{S}(a,b,c)$ cut out our curve $\overline{C}$?
 
 By evaluating the generators of $\mathcal{I}(\overline{C})$ in $(z,xz,\dots,x^az,y,xy,\dots,x^by,1,x, \dots, x^c)$
 one easily finds a set of generators for the ideal of $\overline{C} \cap \TTq^3$. The challenge is now to replace
 this set by two polynomials that are supported on polytopes of the form \[ \Delta_{(a,b,c),b_1} \quad \text{and} \quad \Delta_{(a,b,c),b_2} , \] with 
 $b_1,b_2$ satisfying $b_1 + b_2 = g-5$. Here
 our approach would be to use a Euclidean type of algorithm to find generators whose
 Newton polytopes are as small as possible.
\end{enumerate}

\paragraph*{Point counting timings} We have not implemented anything of the foregoing recipe, but we
can predict how its output should perform in composition with the point counting algorithm from~\cite{tuitman1,tuitman2}, by simply 
starting from a sufficiently generic pair of polynomials $\overline{f}_{\overline{Y}}, \overline{f}_{\overline{Z}} \in \FF_q[x,y,z]$ that are supported on
$\Delta_{(a,b,c),b_1}$ and $\Delta_{(a,b,c),b_2}$
for non-negative integers $a,b,c$ satisfying \eqref{scrollarconditions} and $b_1 + b_2 = g-5$. 
Then one can naively lift to $\mathcal{O}_K[x,y,z]$, take the resultant with respect to $z$, make the outcome
monic using \eqref{mademonic}, and feed the result to the point counting algorithm. 
The tables below contain point counting timings and memory usage
for randomly chosen such pairs in genera $g = 6,7$.
For the sake of conciseness
it makes sense to restrict to the case where the scrollar invariants
$a,b,c$ and the tetragonal invariants $b_1, b_2$ are as balanced as possible, meaning that $c-a \leq 1$ and $b_1 - b_2 \leq 1$,
because this is the generic case~\cite{ballico,bopphoff}${}^\dagger$. We expect the other cases to run faster.\\


\vspace{0.1cm}

\noindent \begin{minipage}[b]{8cm}

\noindent \textbf{$\mathbf{g=6}$}\\

\noindent \scriptsize
\tabcolsep=0.11cm
\begin{tabular}{r||r|r}
             & time    & space   \\
$p$          & pcc(s)  & (Mb)    \\
\hline \hline
$11$         & $8.5$   & $32$    \\
$67$         & $34.7$  & $64$    \\
$521$        & $445$   & $379$   \\
$4099$       & $4748$  & $2504$    
\end{tabular}
\quad
\begin{tabular}{r||r|r}
             & time      & space     \\
$q$          & pcc(s)    & (Mb)      \\
\hline \hline
$3^5$        &  $266$    & $214$    \\
$7^5$        &  $549$    & $325$    \\
$3^{10}$     &  $2750$   &$6072$    \\
$7^{10}$     &  $6407$   &$9814$    
\end{tabular}

\end{minipage}
\hfill
\begin{minipage}[b]{8cm}

\noindent \textbf{$\mathbf{g=7}$}\\


\noindent \scriptsize
\tabcolsep=0.11cm
\begin{tabular}{r||r|r}
             & time     & space   \\
$p$          & pcc(s)   & (Mb)    \\
\hline \hline
$11$         & $11$     & $32$    \\
$67$         & $46$     & $80$    \\
$521$        & $445$    & $347$    \\
$4099$       & $4350$   & $2441$     
\end{tabular}
\quad
\begin{tabular}{r||r|r}
             & time      & space   \\
$q$          & pcc(s)    & (Mb)    \\
\hline \hline
$3^5$        &  $254$    & $156$    \\
$7^5$        &  $550$    & $241$    \\
$3^{10}$     &  $2347$   &$3606$    \\
$7^{10}$     &  $5819$   &$5724$    
\end{tabular}

\end{minipage}

\vspace{5mm}
\noindent \textsc{Laboratoire Painlev\'e, Universit\'e de Lille-1}\\
\noindent \textsc{Cit\'e Scientifique, 59\,655 Villeneuve d'Ascq cedex, France}\\
\vspace{-0.4cm}

\noindent \textsc{Departement Elektrotechniek, KU Leuven and imec-Cosic}\\
\noindent \textsc{Kasteelpark Arenberg 10/2452, 3001 Leuven, Belgium}\\
\vspace{-0.4cm}

\noindent \emph{E-mail address:} \href{mailto:wouter.castryck@gmail.com}{wouter.castryck@gmail.com}\\

\noindent \textsc{Departement Wiskunde, KU Leuven}\\
\noindent \textsc{Celestijnenlaan 200B, 3001 Leuven, Belgium}\\
\vspace{-0.4cm}

\noindent \emph{E-mail address:} \href{mailto:jan.tuitman@wis.kuleuven.be}{jan.tuitman@wis.kuleuven.be}\\

\end{document}